\documentclass[10pt,oneside,reqno,fleqn]{amsart}
\usepackage{geometry}                
\usepackage{graphicx}
\usepackage{caption}
\usepackage{subcaption}
\usepackage{amssymb}
\usepackage{amsmath}
\usepackage{amsthm}
\usepackage{epstopdf}
\usepackage{tabu}
\usepackage{caption}
\usepackage{subcaption}
\usepackage{color}
\usepackage{multirow}

\usepackage[toc, page]{appendix}

\DeclareMathOperator{\Null}{Null}
\newtheorem{theorem}{Theorem}[section]
\newtheorem{lemma}[theorem]{Lemma}

\newtheorem{corollary}[theorem]{Corollary}

\theoremstyle{remark}
\newtheorem{remark}[theorem]{Remark}

\newcommand{\bi}{\begin{itemize}}
\newcommand{\ei}{\end{itemize}}
\newcommand{\ben}{\begin{enumerate}}
\newcommand{\een}{\end{enumerate}}
\newcommand{\be}{\begin{equation}}
\newcommand{\ee}{\end{equation}}
\newcommand{\bea}{\begin{eqnarray}} 
\newcommand{\eea}{\end{eqnarray}}
\newcommand{\ba}{\begin{align}} 
\newcommand{\ea}{\end{align}}
\newcommand{\bse}{\begin{subequations}} 
\newcommand{\ese}{\end{subequations}}
\newcommand{\bc}{\begin{center}}
\newcommand{\ec}{\end{center}}
\newcommand{\bfi}{\begin{figure}}
\newcommand{\efi}{\end{figure}}
\newcommand{\half}{\mbox{\small $\frac{1}{2}$}}
\DeclareMathOperator{\re}{Re}
\DeclareMathOperator{\im}{Im}
\newcommand{\mpspack}{{\tt MPSpack}}       
\newcommand{\eps}{\varepsilon}

\title[Drum problem via the Fredholm determinant]{
Robust and efficient solution of the drum problem
via Nystr\"om approximation of the Fredholm determinant}
\author{Lin Zhao and Alex Barnett}
\address{Department of Mathematics, Dartmouth College, Hanover, NH, 03755}
\date{\today}

\begin{document}
\begin{abstract}

The ``drum problem''---finding the eigenvalues and eigenfunctions of
the Laplacian with Dirichlet boundary condition---has many applications,
yet remains challenging for general domains
when high accuracy or high frequency is needed.
Boundary integral equations 
are appealing for large-scale problems,
yet certain difficulties have limited their use.
We introduce two ideas to remedy this:
1) We solve the resulting nonlinear eigenvalue problem
using Boyd's method for analytic root-finding applied to the
Fredholm determinant.
We show that this is
many times faster than the usual iterative
minimization of a singular value.
2) We fix the problem of spurious {\em exterior resonances}
via a combined-field representation.
This also provides the first robust boundary integral eigenvalue method
for non-simply-connected domains.
We implement the new method in two dimensions using
spectrally accurate Nystr\"om product quadrature.
We prove exponential convergence of the determinant at roots
for domains with analytic boundary.
We demonstrate 13-digit accuracy, and improved
efficiency, in a variety of domain shapes including
ones with strong exterior resonances.

\end{abstract}
\maketitle

\section{Introduction}

Eigenvalue problems (EVPs) for linear PDEs have a wealth of applications
\cite{babosrev}
to modeling vibration problems, acoustic, electromagnetic and quantum cavity resonances,
as well as in modern areas such as nano-scale devices \cite{qdots},
micro-optical resonators for high-power lasers \cite{hakan05},
accelerator design \cite{akcelik}, and
data analysis \cite{saito}.
The paradigm is the Dirichlet eigenvalue problem:
given a bounded connected
domain $\Omega \subset \mathbb{R}^2$ with boundary $\Gamma$,
to find eigenvalues $\kappa^2$ and corresponding nontrivial eigenfunctions $u$
that satisfy
\bea
(\Delta + \kappa^2)u &=& 0 \qquad \mbox{ in }  \Omega,
\label{helm}
\\
u &=& 0 \qquad \mbox{ on } \Gamma,
\label{bc}
\eea
where $\Delta := \partial^2/\partial x_1^2 + \partial^2/\partial x_2^2$
is the Laplacian.
We refer to $\kappa$ as the eigenfrequency,
and label the allowable set $\kappa_1<\kappa_2\le \kappa_3 \le \cdots
\nearrow +\infty$,
counting multiplicities.
$u_j$ will refer to an eigenfunction for the eigenfrequency $\kappa_j$.
A numerical solution is necessary in all but a few special shapes
(in 2D, ellipses and rectangles) where the Laplacian is separable
\cite{CoHi53}.
This and related EVPs are also of interest in mathematical areas such as
quantum chaos \cite{nonnenrev}.
This is covered in
excellent reviews by Kuttler--Sigillito \cite{KS} and Grebenkov--Nguyen
\cite{grebenkov}.

Numerical solution of \eqref{helm}--\eqref{bc}
falls broadly into two categories:
A) direct discretization, using finite differencing or finite elements
to give a sparse linear EVP where $\kappa^2$ is the eigenvalue;
vs
B) reformulation as a boundary integral equation (BIE)
\cite{backerbim},
discretized using the Galerkin or Nystr\"om methods,
resulting in a highly {\em nonlinear} EVP, again for the eigenvalue $\kappa^2$
or eigenfrequency $\kappa$.
The nonlinearity with respect to $\kappa$
comes from that of the fundamental solution to the Helmholtz
equation \eqref{helm}.
The advantages of type B include:
a huge reduction in the number of unknowns
(due to the decrease in dimensionality by one) especially at high frequency,
and
increased accuracy (since finite element high-frequency
``pollution'' \cite{pollution}
is absent). High-order or spectral accuracy is not hard to achieve,
at least in two dimensions (2D).
Yet, as pointed out by B\"acker \cite[Sec.~3.3.6]{backerbim}, the standard
BIE method
is not even robust for a simply-connected domain,
due to the possibility of spurious exterior resonances.
Recently, type-B methods which approximately {\em linearize}
the nonlinearity, hence boost efficiency at high frequency,
have been created,
but these are limited to moderate $\kappa$ \cite{kirkup},
to heuristic methods with low accuracy \cite{hakansca,veblemonza},
or to domains that are star-shaped \cite{v+s,que,mush,sca}.
This motivates the need for a robust type-B method that
applies to all domain shapes, including multiply-connected ones,
and remains efficient up to at least medium-high frequencies.

In this work we solve two of these issues:
(1)
The standard approach to solve the nonlinear EVP is
by searching for ``V-shaped'' minima of a smallest singular value
\cite{tref06};
we boost efficiency by
turning this into a search for the roots of an {\em analytic} function,
which can be done with less function evaluations and without the
expensive computation of the SVD.
(2) We solve the exterior resonance problem, and at the same
time the case of multiply-connected domains,
using a combined field integral equation \eqref{CFIE}.
We also provide several analysis results that place our method
on a rigorous footing.

The outline of this paper is as follows. In Section~\ref{s:bie}
we review the use of potential theory to reformulate 
the eigenvalue problem as a BIE,
and give a discretization of the BIE due to Kress
\cite{kress91} that achieves spectral accuracy for smooth domains.
To tackle issue (1) above,
in Section \ref{s:fred} we introduce the {\em Fredholm determinant},
\be
f(\kappa) = \det ( I - 2 D(\kappa))
\label{fk}
\ee
where $D$ is the double-layer operator (defined by \eqref{dlp} below),
whose roots are precisely the eigenfrequencies $\kappa_j$ for simply connected domains. 
Following Bornemann \cite{bornemann},
we approximate this with the determinant of a Nystr\"om matrix.
Our main Theorem~\ref{t:main},
in Section~\ref{s:thm}, states that
this approximation convergences exponentially to zero at the true
eigenfrequencies, if the domain has analytic boundary.
Since $f(\kappa)$ is analytic for $\kappa$ nonzero, we
propose in Section~\ref{s:boyd} applying Boyd's method to find its roots,
an application we have not seen in the literature before.
In Section ~\ref{s:res} we prove, and demonstrate numerically,
that the CFIE \eqref{CFIE}
is robust for domains with exterior resonances or interior holes.
It is well known that finding roots becomes ill-conditioned when they are close,
hence we explain in Section~\ref{s:svd}
how we retain robustness in the case of nearby eigenvalues
by reverting to the (more expensive) SVD method in these (rare) case.
Section \ref{s:c} gives numerical performance tests of the entire scheme,
achieving 13 digits the first 100 eigenfrequencies of
a general domain, and a domain with exterior resonances, showing that
our method is competitive in terms of both accuracy and timing.

\section{Boundary integral formulation and quadrature scheme}
\label{s:bie}
Now we lay the foundation of our method for computing eigenfrequencies
by describing the boundary integral formulation and its analyticity
properties for analytic domains,
and then its numerical treatment in 2D, which is standard \cite{kress91}.
\subsection{Integral equation formulation}

For a bounded domain $\Omega$ with twice continuously differentiable boundary $\Gamma$, we explicitly construct solutions to the Helmholtz equation by layer potentials using the fundamental solution as the kernel. The fundamental solution is given by
\begin{equation}
\Phi(x,y):= \frac{i}{4}H^{(1)}_{0}(\kappa \vert x-y \vert), \hspace{0.5cm} x \neq y,
\qquad x,y\in\mathbb{R}^2,
\end{equation}
where $H^{(1)}_{0}$ is the first-kind Hankel function of order zero.

For a continuous function $\varphi$ on $\Gamma$, the single layer operator $\mathcal{S} : C(\Gamma) \rightarrow C(\mathbb{R}^2 \backslash \Gamma)$ is defined as follows, with $v$ denoting the resulting single layer potential
\begin{equation}
v(x) = \mathcal{S} \varphi (x):= \int_{\Gamma} \Phi(x,y) \varphi(y)ds(y), \hspace{0.5cm} x\in \mathbb{R}^{2} \backslash \Gamma,
\end{equation}
where $ds(y)$ is the arc-length element on $\Gamma$.
Note that the domain of $v$ excludes $\Gamma$,
The corresponding boundary operator $S : C(\Gamma) \rightarrow C(\Gamma) $ is
\begin{equation}
S \varphi (x):= \int_{\Gamma} \Phi(x,y) \varphi(y)ds(y), \hspace{0.5cm} x\in \Gamma.
\end{equation}
The double layer operator $\mathcal{D}: C(\Gamma) \rightarrow C(\mathbb{R}^2 \backslash \Gamma)$ with associated double layer potential $u$ is given by
\begin{equation}
u(x) = \mathcal{D} \varphi  (x):= \int_{\Gamma} \frac{\partial\Phi(x,y)}{\partial n(y)} \varphi(y)ds(y), \hspace{0.5cm} x \in \mathbb{R}^{2} \backslash \Gamma.
\label{DLP}
\end{equation}
where $n(y)$ is the unit normal vector at $y\in\Gamma$ directed to the exterior of the domain. Again, because the integral exists for $x \in \Gamma$, one may define
a boundary operator $D: C(\Gamma) \rightarrow C(\Gamma)$ by
\begin{equation}
D \varphi  (x):= \int_{\Gamma} \frac{\partial\Phi(x,y)}{\partial n(y)} \varphi(y)ds(y), \hspace{0.5cm} x \in \Gamma.
\label{dlp}
\end{equation}
The above operators depends on the frequency $\kappa$, and we will
indicate this only when needed.

Both $u$ and $v$ as defined are solutions to the Helmholtz equation and can be continuously extended, either from the interior or the exterior of $\Omega$,  to the boundary by taking limits in the following sense:
\be
u^{\pm}(x):=\lim_{h\rightarrow 0+} u(x \pm hn(x)),
\qquad
u_n^\pm(x):=\lim_{h\rightarrow 0+} n(x) \cdot \nabla u(x \pm hn(x)),
\qquad x \in \Gamma,
\ee
and analogously for $v$.
These limits relate to the boundary operators via the jump relations
\cite{coltonkress}
\bea
v^{\pm}(x) = S \varphi (x), \hspace{0.5cm} x \in \Gamma,
\label{JR1}
\\
v_n^{\pm}(x) = (D^T \mp \half) \varphi (x), \hspace{0.5cm} x \in \Gamma,
\label{JR2}
\\
u^{\pm}(x)= (D \pm \half) \varphi(x), \hspace{0.5cm} x \in \Gamma.
\label{JR3}
\\
u_n^{\pm}(x) = T \varphi (x), \hspace{0.5cm} x \in \Gamma,
\label{JR4}
\eea
where $D^{T}$ is given by
\begin{equation}
D^T \varphi  (x):= \int_{\Gamma} \frac{\partial\Phi(x,y)}{\partial n(x)} \varphi(y)ds(y), \hspace{0.5cm} x \in \Gamma.
\end{equation}
and the hypersingular operator $T$ is defined by
\begin{equation}
T \varphi  (x):= \frac{\partial}{\partial n(x)}\int_{\Gamma} \frac{\partial\Phi(x,y)}{\partial n(y)} \varphi(y)ds(y), \hspace{0.5cm} x \in \Gamma,
\end{equation}

When $u$ is given by a double-layer potential with density $\varphi$,
enforcing the Dirichlet boundary condition \eqref{bc} gives
\begin{equation}
\label{eq: BIEwD}
(I -2D(\kappa)) \varphi = 0 .
\end{equation}
Thus we might hope that the (nonlinear) eigenvalue problem that
$I-2D(\kappa)$ has a nontrivial nullspace is equivalent to
the (linear) eigenvalue problem \eqref{helm}-\eqref{bc}.
For a domain of general connectivity there is {\em not} such an equivalence;
we merely have the following.
\begin{lemma}
Let $\Omega$ be a (possibly non-simply connected)
bounded domain with twice continuously differentiable boundary $\Gamma$.
Then if $\kappa^2$ is a Dirichlet eigenvalue of $\Omega$,
$I-2D(\kappa)$ has a nontrivial nullspace.
\label{l:nullspace}
\end{lemma}
\begin{proof}
Green's representation theorem \cite[Theorem 2.1]{coltonkress} states
that if $(\Delta +\kappa^2)u=0$ in $\Omega$, then
\be
{\mathcal S}u_n^- - {\mathcal D} u^- \; = \;
\left\{\begin{array}{ll} u, & \qquad \mbox{ in } \Omega ~, \\
0, & \qquad
\mbox{ in } \mathbb{R}^2\backslash\overline{\Omega}~.\end{array}\right.
\label{GRF}
\ee
Applying this to $u$ an eigenfunction at frequency $\kappa_j$, and
taking its derivative on $\Gamma$ using \eqref{JR2} gives
$u_n^- = (D^T + \half)u_n^-$. Since $u_n^-$ is nontrivial,
the compactness of $D$ and the Fredholm alternative proves 
$I -2D(\kappa_j)$ has a nontrivial nullspace.
\end{proof}
The consequence for a multiply-connected domain is that it is
possible that there are {\em spurious} frequencies where $I-2D(\kappa)$ has
a nontrivial nullspace but $\kappa^2$ is not a Dirichlet eigenvalue
(we will characterize these frequencies in Lemma~\ref{l:hole}).

Only for the case of $\Omega$ simply-connected does
equivalence hold, as the following well-known theorem states.
\begin{theorem}\cite{CK83}     
\label{kress}
Let $\Omega$ be a bounded, simply-connected
domain with twice continuously differentiable boundary $\Gamma$.
Then for each $\kappa \in \mathbb{C}\backslash\{0\}$ with $\operatorname{Im}\kappa \geq 0$, $\kappa^2$ is a Dirichlet eigenvalue of $\Omega$
if and only if $I-2D(\kappa)$ has a nontrivial nullspace.
Moreover, the dimension of the eigenspace is the same as that of the nullspace.
\end{theorem}     
For the case of Lipschitz boundary, see Mitrea \cite{mitrea}.
This motivates integral equations as a robust approach to the Dirichlet
eigenvalue problem for simply-connected domains;
later in Section~\ref{s:res} we will show how to handle
multiply-connected domains.

\subsection{Splitting of the kernel}

We will discuss a quadrature scheme for Helmholtz kernels
that is highly accurate for smooth boundaries \cite{kress91};
for this an analytic splitting is needed.
Assume $\Gamma$ is analytic and has a regular parametrization $x(t) = (x_{1}(t), x_{2}(t)), 0 \leq t \leq 2\pi$. We transform \eqref{eq: BIEwD} into the parametric form
\begin{equation}
\label{paraBIE}
\psi (t) - \int_{0}^{2\pi}L(t,s) \psi (s) ds = 0, \hspace{0.5cm} 0 \leq t \leq 2\pi
\end{equation}
where $\psi (t):= \varphi(x(t))$ and the kernel of the reparametrized
operator $2D(\kappa)$ is given by
\begin{eqnarray}
L(t,s) &:=& 
\frac{\partial\Phi(x(t),x(s))}{\partial n(x(s))} |x'(s)|
\label{L}
\\
&=&\frac{i\kappa}{2} \{ x'_{2}(s)[x_{1}(t)-x_{1}(s)]-x'_{1}(s)[x_{2}(t)-x_{2}(s)]\}\frac{H_{1}^{(1)}(\kappa r(t,s))}{r(t,s)} 
\end{eqnarray}
with the distance function $r(t,s) := \{ [x_{1}(t) - x_{1}(s)]^{2} + [x_{2}(t) - x_{2}(s)]^{2}\}^{\frac{1}{2}}$.

With a slight abuse of notation, at each $\kappa$ we use $L(\kappa): C[0, 2\pi] \rightarrow C[0, 2\pi]$ to denote the integral operator with $L(t,s)$ as its kernel, that is, the reparametrized operator $2D$.
We will sometimes drop the explicit dependence on $\kappa$ and write $L$.
The kernel $L$ is continuous but not analytic,
so one splits the kernel into
\begin{equation}
L(t,s) = L^{(1)}(t,s) \ln{\bigl(4\sin^2{\frac{t-s}{2}}\bigr)} + L^{(2)}(t,s)
\end{equation}
where
\begin{equation}
L^{(1)}(t,s):=-\frac{\kappa}{2\pi} \{ x'_{2}(s)[x_{1}(t)-x_{1}(s)]-x'_{1}(s)[x_{2}(t)-x_{2}(s)]\}\frac{J_{1}(\kappa r(t,s))}{r(t,s)}
\label{L1}
\end{equation}
\begin{equation}
L^{(2)}(t,s):=L(t,s)-L^{(1)}(t,s)\ln \bigl(4\sin^2{\frac{t-s}{2}}\bigr)
\label{L2}
\end{equation}
Both $L^{(1)}$ and $L^{(2)}$ are analytic, provided that $\Gamma$ is analytic
\cite{kress91}.
In that case we get the following.
\begin{lemma}
Let $\Omega$ have analytic boundary.
Then any density function $\psi(s)$ solving \eqref{paraBIE}
is an analytic function of the parameter $s$.
\label{l:anal}
\end{lemma}
This follows from the argument
of \cite[Prob.\ 12.4, p.~217]{LIE}, namely that the operator $L$ is compact in the space of $2\pi$-periodic analytic
functions in a complex strip $\mathbb{R}\times (-a,a)$ for some $a>0$,
and the Fredholm alternative.

\subsection{Quadrature and Nystr\"om method}
\label{s:quad}

We choose a set of quadrature points equidistant in parameter,
$s_{k} := \frac{2\pi k}{N}$, $k = 0,1,..., N-1$,
where $N$ is an even number, with equal weights $2\pi/N$,
and insert this quadrature into \eqref{paraBIE}
to get the approximation
\begin{equation}
\label{eq: FDAppr}
\psi^{(N)}(t) - \sum_{k=0}^{N-1} \{ R^{(N)}_{k}(t)L^{(1)}(t, s_{k}) + \frac{2\pi}{N} L^{(2)}(t, s_{k})\} \psi^{(N)}(s_{k}) = 0, \hspace{0.5cm} 0\leq t \leq 2\pi.
\end{equation}
Here the second term inside the curly brackets arises from the usual quadrature
rule, whereas the first term arises from a spectrally-accurate
product quadrature
scheme for the periodized log singularity (reviewed in \cite[Sec.~6]{hao}),
with weights
\begin{equation}
R^{(N)}_k(t) = -\frac{4\pi}{N} \sum_{m=1}^{\frac{N}{2}-1}\frac{1}{m}\cos{m(t-s_{k})} - \frac{4\pi}{N^{2}} \cos{\frac{N}{2}(t-s_{k})}, \hspace{0.5cm}
k = 0,\ldots,N-1
\end{equation}
Define $L_N$ to be the Nystr\"om interpolant from \eqref{eq: FDAppr},
which maps $\psi \in C[0, 2\pi]$ to
\begin{equation}
\label{defl}
L_{N} \psi (t) = \sum_{k=0}^{N-1} \{ R^{(N)}_{k}(t)L^{(1)}(t, s_{k}) + \frac{2\pi}{N} L^{(2)}(t, s_{k})\} \psi(s_{k})
\end{equation}
Kress \cite[Sec.~12.3]{LIE}
showed that, when $L^{(1)}$ and $L^{(2)}$ have analytic kernels,
interpolation of analytic functions
with this product quadrature convergences exponentially with $N$ in the $L_{\infty}$-norm.
Thus for each analytic $\psi$, $\Vert L_{N} \psi - L\psi \Vert _{\infty} \leq Ce^{-aN}$ for some constants $C$ and $a$ depending on $\psi$ \cite[p.~185]{LIE}.

By setting $t$ to $s_i$ in \eqref{defl}, one obtains the Nystr\"om matrix $M_N$ with elements
\begin{equation}
\label{nm}
(M_N)_{ij} := R^{(N)}_{|i-j|}(0) L^{(1)}(s_{i},s_{j})+\frac{2\pi}{N}L^{(2)}(s_{i},s_{j})
\qquad i,j = 0,\dots, N-1.
\end{equation}
The condition \eqref{paraBIE} that $I-2D(\kappa)$, and hence $I-L(\kappa)$,
is singular
can now be approximated with exponentially small error
by the condition that the matrix $I-M_N(\kappa)$ is singular.
Each null vector of $I-L_N$ is exactly reconstructed
by applying the interpolant on the right-hand side of \eqref{defl} to the
corresponding null vector of the matrix.
By the analysis in \cite[Sec.~12.2-12.3]{LIE} in the homogeneous case,
this reconstructs the desired null vectors of $I-L$ to exponential
accuracy.

\section{The Fredholm determinant}
\label{s:fred}
As we have seen,
for simply-connected domains,
$\kappa$ is an eigenfrequency if and only if the boundary integral operator $I-L{(\kappa)}$ has nontrivial kernel.
We now convert this to a condition on a Fredholm determinant.

The following theorem says that we can study the invertibility of $I-L$ on $L_{2}[0, 2\pi]$ instead of $C[0, 2\pi]$.
\begin{theorem}\cite{hahner} \cite[p.91]{LIE}
Let $A$ be an integral operator with weakly singular kernel, then the nullspaces of $I-A$ in $C[0, 2\pi]$ and $L_{2}[0, 2\pi]$ coincide.
\end{theorem}
This implies that $L$ has the same set of nonzero eigenvalues, counting multiplicities, in $C[0, 2\pi]$ as in $L_{2}[0, 2\pi]$.
Thus from now on we need not specify in which space we consider these eigenvalues.

Let $\mathcal{J}_{1}(L_{2}[0, 2\pi])$ be the space of {\em trace-class}
operators in $L_{2}[0, 2\pi]$.
This space is defined by finiteness of the
operator norm
$\|A\|_{\mathcal{J}_1}$, which is the sum of the operator singular values
\cite{bornemann}; this insures that the sum of the eigenvalues is also bounded.
\begin{lemma}
$L$ with kernel given by \eqref{L} is a trace-class operator.
\label{l:tr}
\end{lemma}
\begin{proof}
Using the Bessel function asymptotic \cite[10.8.1]{DLMF},
the leading non-analytic term in $L(t,s)$ is
$O((t-s)^2\log |t-s|)$ for small $t-s$,
thus $L(t,s)$ and the partial derivative
$\partial_{s}L(t,s)$ are continuous on $[0, 2\pi]^2$,
thus $L$ is trace class on $L_{2}[0,2\pi]$ \cite{bornemann}.
\end{proof}
For trace-class operators, the Fredholm determinant as a linear functional can be constructed in several equivalent ways; we take the approach of
Gohberg and Krein \cite[p.~157]{GK69}.
Let $\mathcal H$ be a Hilbert space. For $A \in \mathcal{J}_{1}(\mathcal{H)}$ with nonzero eigenvalues $\lambda_{1}(A), \lambda_{2}(A),\ldots$ (counting multiplicities), the Fredholm determinant of $I-A$ is defined by
\begin{equation}
\label{fd1}
\det(I-A) : = \displaystyle \prod_{j=1}^\infty(I-\lambda_{j}(A))
\end{equation}

One important property of the Fredholm determinant is that it completely describes when $I-A$ is invertible:
\begin{theorem}\cite[p.~34]{simon}
\label{invTr}
For  $A \in \mathcal{J}_{1}(\mathcal{H)}$, $\det(I-A)\neq 0$ if and only if $I-A$ is invertible.
\end{theorem}

\begin{corollary}
\label{iffL}
L with kernel given by \eqref{L} satisfies $\det{(I-L)}=0$ if and only if $I-L$ has nontrivial kernel space.
\end{corollary}
\begin{proof}
The third Riesz theorem \cite[page 11]{CK83} says if $I-L$ is not surjective, it is not injective.
The claim follows from Theorem \ref{invTr}.
\end{proof}

As we will see in Section \ref{s:thm}, the nonzero eigenvalues of $L$ will be approximated numerically by the nonzero eigenvalues of $L_{N}$ in $C[0, 2\pi]$. The following lemma connects those nonzero eigenvalues of $L_{N}$ to the ones of the Nystr\"om matrix, making accurate numerical approximation of $\det(I-L)$ possible.

\begin{lemma}
\label{fnm}
The collection of nonzero eigenvalues, counting multiplicities, of $ L_{N}$, defined in \eqref{defl}, is the same as the nonzero eigenvalues of the associated Nystr\"om matrix $M_{N}$ as defined in \eqref{nm}.
\end{lemma}
\begin{proof}
If $\lambda$ is a nonzero eigenvalue of $L_{N}$, then there exists a finite dimensional eigenspace with
basis $\{\varphi_{i}\}$ such that $L_{N}\varphi_{i} = \lambda \varphi_{i}$ holds on $[0, 2\pi]$. Certainly it holds on all the quadrature nodes, 
meaning $M[\varphi_{i}(s_{k})]_{k =0}^{N-1}=\lambda [\varphi_{i}(s_{k})]_{k=0}^{N-1}$,
where $[\varphi_{i}(s_{k})]_{k =0}^{N-1}$ indicates a column vector.
It cannot be true that $\varphi_{i}$ is simultaneously zero at all quadrature nodes, since then by ~\eqref{defl}, $\varphi_{i}$ is identically zero on $[0, 2\pi]$. By the same reasoning the set of $[\varphi_{i}(s_{k})]_{k=0}^{N-1}$ for all $i$
is a linearly independent set of eigenvectors of $M_{N}$ with eigenvalue $\lambda$. 

If on the other hand $\lambda$ is a nonzero eigenvalue of $M_{N}$, then there exists a finite dimensional eigenspace with a basis spanned by the vectors $\{[\phi_{i,k}]_{k=0}^{N-1}\}$. For each $i$ we can construct $\varphi_{i} (t) =  \frac{1}{\lambda}\sum_{k=0}^{N-1} \{ R^{(N)}_{k}(t)L^{(1)}(t, s_{k}) + \frac{2\pi}{N} L^{(2)}(t, s_{k})\} \phi_{i,k}$, then $[\varphi_{i}(s_{k})]_{k=0}^{N-1} = \frac{1}{\lambda} M_{N} [\phi_{i,k}]_{k=0}^{N-1} = [\phi_{i,k}]_{k=0}^{N-1}$. One sees that $\varphi_{i}$ is an eigenfunction of $L_{N}$ with eigenvalue $\lambda$, and $\{\varphi_{i}\}$ is a linearly independent set because the set $\{[\phi_{i,k}]_{k=0}^{N-1}\}$ is.
\end{proof}

The Fredholm determinant is a function of $\kappa$, and we use the
notation \eqref{fk} for the determinant of the exact operator.
Similarly we use, for the matrix determinant of the associated Nystr\"om matrix,
\begin{equation}
\label{fnk}
f_{N}(\kappa) := \det(I -M_{N}{(\kappa)}) ~.
\end{equation}

\begin{remark}
In fact, from Lemma \ref{fnm}, it follows that $f_{N}$ is the Fredholm determinant of $I-L_{N}$ as a finite dimensional operator on $C[0, 2\pi]$. The definition of Fredholm determinant for certain operators on a Banach space and more can be found in \cite{gohberg}.
\end{remark}


\begin{figure}[!ht]
  \centering
   \includegraphics[width=0.4\textwidth]{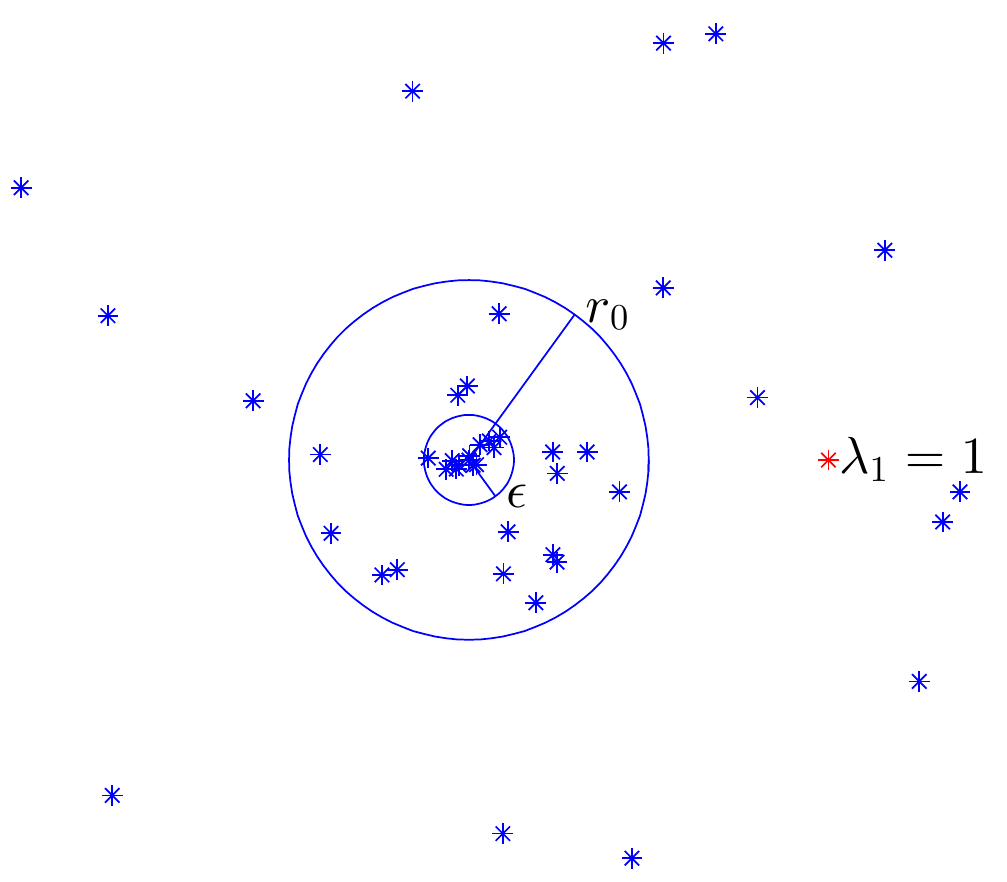}
   \caption{Illustration of proof idea for the main Theorem~\ref{t:main}, showing spectrum of $L{(\kappa_{j})}$, and circles of radius $r_{0}$ and $\epsilon$.}
     \label{fig: proof}
\end{figure}

\section{Error analysis of the Fredholm determinant}
\label{s:thm}
We prove our main error analysis result in this section. The approximation sequence $\{L_{N}\}$ converges pointwise to the integral operator $L$ on $C[0, 2\pi]$, and is collectively compact \cite[p.~202]{LIE}. The following two theorems of Atkinson describe the convergence of eigenvalues of $L_{N}$ to the ones of $L$.

\begin{theorem}\cite{atkinspec}
\label{Atkinson1}
Let $K$ be an integral operator on a Banach space and $\{K_N\}$ be a collectively compact sequence of numerical integral operators approximating $K$ pointwise, and let $R$ and $\epsilon$ be arbitrary small positive numbers. Then there is an $N_{0}$ such that for $N \geq N_{0}$, any eigenvalue $\lambda$ of $K_{N}$ satisfying $\vert \lambda \vert \geq R$ is within $\epsilon$ of an eigenvalue $\lambda_{0}$ of $K$ with $\vert \lambda_{0} \vert \geq R$. Furthermore let $\sigma_{N}$ be the set of eigenvalues of $K_{N}$ within distance $\epsilon$ from a fixed $\lambda_{0}$, then the sum of multiplicities of $\lambda$ in $\sigma_{N}$ equals the multiplicity of $\lambda_{0}$.
\end{theorem}

To summarize, outside of an arbitrarily small disk eigenvalues of $L_{n}$ approximate the eigenvalues of $L$ with correct multiplicities. We also have a guarantee that the convergence rate of $L_{n}$ is carried over to the eigenvalues.

\begin{theorem}\cite{atkinrate}
\label{Atkinson2}
With the same assumption as in the above theorem, let $\lambda_{0}$ be of index $\nu$, i.e., $\nu$ is the smallest integer for which
\begin{equation}
\ker( (\lambda_{0}-K)^{\nu} )= \ker((\lambda_{0}-K)^{\nu+1}),
\end{equation}
where ker means the kernel space.
Then for some $c >0$ and all sufficiently large $n$, 
\begin{equation}
\vert \lambda - \lambda_{0} \vert \leq c \max \{ \Vert K \varphi_{i} -K_{n} \varphi_{i} \Vert ^{\frac{1}{\nu}} \vert 1 \leq i \leq m\}
\end{equation}
for all $\lambda \in \sigma_{n}$, and the set $\{ \varphi_{1},...\varphi_{m}\}$ is a basis for $\ker ((\lambda_{0}-K)^{\nu})$.
\end{theorem}

When $\lambda_{0} =1$, then index is also called the Riesz number \cite[p.11]{CK83} of $K$. The Riesz number of our $L$ is 1 \cite[p.84]{CK83}.  We can now
prove the main theorem that the determinant of $I-M_{N}{(\kappa)}$ at an eigenfrequency $\kappa=\kappa_{j}$ vanishes exponentially with $N$.

\begin{theorem} 
Let $\kappa_{j}^{2}$ be a
Dirichlet eigenvalue of a bounded domain $\Omega$ with analytic boundary. Then
there exists an $N_{0}$ such that
\be
\vert f_N(\kappa_{j}) \vert \leq Ce^{-\alpha N}
\qquad \mbox{for all } N>N_0,
\ee
where $C$ and $\alpha>0$ are constants depending on $\Omega$ and $\kappa_{j}$.
\label{t:main}
\end{theorem}

\remark{This theorem includes the case of $\Omega$ non-simply connected,
although later we will show that a modification to the
definition of $f(\kappa)$ and $f_N(\kappa)$ is needed to
make a robust method for this case.}

\begin{proof}
Let $\{ \lambda_{i}^{(N)}\}_{i=1}^{N'}$
be the set of nonzero eigenvalues of $L_N(\kappa_j)$, counting multiplicities,
where $N'$ is at most $N$.
Let $\{\lambda_{i}\}$ be nonzero eigenvalues of $L{(\kappa_{j})}$.
If $\kappa_{j}$ is an eigenfrequency for \eqref{helm}--\eqref{bc},
then according to Lemma~\ref{l:nullspace}, $I - L{(\kappa_{j})}$ has
nontrivial kernel. Based on Corollary~\ref{iffL}, $1$ is an eigenvalue of $L$, which we can label
$\lambda_{1}=1$.
Theorem \ref{Atkinson1} implies that we can pick an ordering of $
\lambda_{i}^{(N)}$ so that $\{\lambda_{1}^{(N)}\}$ converges to
$\lambda_{1}$ as $N \rightarrow \infty$, and there might be multiple
such sequences, depending on the multiplicity of $\lambda_{1}$, i.e.,
essentially, the number of sequences with $\lambda_{1}$ as the limit
is the same as the multiplicity of $\lambda_{1}$. We only need the
existence of one such sequence for the following proof to hold.
Theorem \ref{Atkinson1} also implies that if we let $r_{0}$ be
a constant with $r_{0} \leq \frac{1}{2}$, then there exists $N_{1}$
such that for $N \geq N_{1}$, the number of $\lambda_{i}^{(N)}$ with
$\vert \lambda_{i}^{(N)} \vert \geq r_{0}$ equals the number of
$\lambda_{i}$ with $\vert \lambda_{i} \vert \geq r_{0}$, and we can
relabel $\{\lambda_{i}^{(N)}\}_{i=1}^{N'}$ in such a way that $\displaystyle
\lim_{N \rightarrow \infty} \lambda_{i}^{(N)} = \lambda_{i}$ for all
$i$ with $\lambda_{i}^{(N_{1})} \geq r_{0}$.

Since by Lemma \ref{fnm}, $\{ \lambda_{i}^{(N)} \vert 1\leq i \leq N' \}$
is also the set of
nonzero eigenvalues of $M_{N}{(\kappa_{j})}$, we write the matrix determinant of $I - M_{N}{(\kappa_{j})}$ as a product of three factors as follows,
\be
\det(I - M_{N}{(\kappa_{j})})  \;=\; \prod _{i=1}^{N'} (1 - \lambda_{i}^{(N)})
                         \;=\; (1 - \lambda_{1}^{(N)}) \prod_{\vert \lambda^{(N)}_{i} \vert \geq r_{0} \mbox{, } i \neq 1} (1 - \lambda_{i}^{(N)}) \prod_{\vert \lambda^{(N)}_{i} \vert < r_{0}} (1 - \lambda_{i}^{(N)})
\ee
Then, for the second factor, since there are a finite number of terms,
\be
\lim_{N\rightarrow \infty} \prod_{\vert \lambda^{(N)}_{i} \vert \geq r_{0} \mbox{, } i \neq 1} (1 - \lambda_{i}^{(N)}) 
\; = \;
\prod_{\vert \lambda^{(N)}_{i} \vert \geq r_{0} \mbox{, } i \neq 1} \lim_{N \rightarrow \infty} (1 - \lambda_{i}^{(N)})
 \; = \;
\prod_{\vert \lambda_{i} \vert \geq r_{0} \mbox{, } i \neq 1} (1 - \lambda_{i})
\ee
Thus there exists $N_{2}$ and constant $C_{1}$ such that for $N \geq N_{2}$, 
\begin{equation}
\biggl\vert \prod_{\vert \lambda^{(N)}_{i} \vert \geq r_{0} \mbox{, } i \neq 1} (1 - \lambda_{i}^{(N)} ) \biggr\vert  \leq C_{1}
\nonumber
\end{equation}

For the third factor 
\be
\prod_{\vert \lambda^{(N)}_{i} \vert < r_{0}} \vert1 - \lambda_{i}^{(N)} \vert
\;=\;
\exp (\sum_{\vert \lambda^{(N)}_{i} \vert < r_{0}}  \log \vert1 - \lambda_{i}^{(N)} \vert )
\; \le \;
  \exp (\sum_{\vert \lambda^{(N)}_{i} \vert < r_{0}}  2 \vert  \lambda_{i}^{(N)} \vert )
\ee
Choose $\epsilon \in (0,r_0)$, and
let $m_{N}(\epsilon)$, $m_{N}(r_{0})$ be the number of $\lambda^{(N)}_i$ with $\vert \lambda^{(N)}_i \vert \geq \epsilon$,
$\vert \lambda^{(N)}_i \vert \geq r_{0}$, respectively; see Fig.~\ref{fig: proof}.
In Theorem~\ref{Atkinson1}, pick $R = \epsilon$, then there exists $N_{3}$ such that for $N \geq N_{3}$, all $ \lambda^{(N)}_{i} $ with $\vert \lambda_{i}^{(N)} \vert \geq \epsilon$ are within distance $\epsilon$ of some $\lambda_{i}$,
and each $\lambda_{i}$ with $\vert \lambda_{i} \vert \geq \epsilon$ has exactly one sequence $\{ \lambda_{i}^{(N)} \}$ approaching it, i.e., we have $|\lambda_i^{(N)}-\lambda_i|<\epsilon$ and $\vert \lambda_{i}^{(N)} \vert \geq \epsilon$ for $N \geq N_{3}$.
Then for $N\ge N_3$, we bound
\begin{align*}
\sum_{\vert \lambda^{(N)}_{i} \vert < r_{0}}  \vert  \lambda_{i}^{(N)} \vert & = \sum_{\vert \lambda^{(N)}_{i} \vert < \epsilon}  \vert  \lambda_{i}^{(N)} \vert + \sum_{\epsilon \leq \vert \lambda^{(N)}_{i} \vert < r_{0}}  \vert  \lambda_{i}^{(N)} \vert \\
                                            & \leq (N' - m_{N}(\epsilon)  )\epsilon + \sum_{\epsilon \leq \vert \lambda_{i} \vert < r_{0}}  \vert  \lambda_{i} \vert + (m_{N}(\epsilon) -m_{N}(r_{0}))\epsilon  \\
                                            &  = ( N'-  m_{N}(r_{0})) \epsilon + \sum_{\epsilon \leq \vert \lambda_{i} \vert < r_{0}}  \vert  \lambda_{i} \vert  \\
                                            &  \leq N' \epsilon + \sum_{\epsilon \leq \vert \lambda_{i} \vert < r_{0}}  \vert  \lambda_{i} \vert
\end{align*}
where
\begin{equation}
 \sum_{\epsilon \leq \vert \lambda_{i} \vert < r_{0}}  \vert  \lambda_{i} \vert  \leq \sum_{i} \vert \lambda_{i} \vert \leq  \Vert L{(\kappa_{j})} \Vert _{\mathcal{J}_{1}}
\nonumber
\end{equation}
which is bounded since by Lemma~\ref{l:tr} $L$ is in trace class.

For $\varphi \in \ker(I-L)$, from Lemma~\ref{l:anal}, $\varphi$ is analytic thus,
as discussed in Section~\ref{s:quad},
our quadrature scheme has $\Vert L_{N} \varphi - L \varphi \Vert _{\infty} \leq Ce^{-a_{0}N}$ for $N$ sufficiently large, where $a_{0}>0$ and $C$ are constants which only depend on $\varphi$. $\ker(I-L)$ is finite dimensional so by theorem \ref{Atkinson2}, so there exists $N_{4}$, $a >0$ and $C_{2}$ such that for $N \geq N_{4}$, $\vert 1-\lambda_{1}^{(N)}\vert \leq C_{2} e^{-a N}$,

Let $N_{0}= \max \{N_{1}, N_{2}, N_{3}, N_{4}\}$ then for $N \geq N_{0}$,
since $N'\le N$,
\be
\vert \det(I - M_{N}{(\kappa_{j})}) \vert 
\;\leq\;
C_{2} e^{-aN} C_{1} \exp (2N\epsilon +2  \Vert L{(\kappa_{j})} \Vert _{\mathcal{J}_{1}}) 
\ee
Now let $C := C_{1}C_{2}\exp(2 \Vert L{(\kappa_{j})} \Vert _{\mathcal{J}_{1}})$,
then
$\vert \det(I - M_{N}^{(\kappa_{j})}) \vert  \leq C e^{-(\alpha-2 \epsilon)N}$,
so we may choose any positive $\alpha<a - 2\epsilon$ to finish the proof.
\end{proof}

\remark{From the above proof, it is clear that when $\ker(I-L)$ is one-dimensional the rate $\alpha$
may be chosen arbitrarily close to $a$, the width of the
strip in which the null-vector $\varphi$ (density generating the
eigenfunction) is analytic.
Similar result holds for  $\ker(I-L)$ higher-dimensional. }

\remark{When the boundary $\Gamma$ is merely $C^\infty$ smooth
(not necessarily analytic), we expect that
$\ker(I-L)$ is in $C^{\infty}[0, 2\pi]$, and that the
determinant converges to zero super-algebraically at eigenfrequencies.
We leave a proof of this to future work.}

\section{Boyd's method for finding roots of the determinant}
\label{s:boyd}

Here we describe a new approach to finding eigenvalues efficiently,
using Theorem \ref{invTr} to equate these with the roots of the
Fredhold determinant $f(\kappa)$.
Our method is inspired by the following fact.
\begin{lemma}
$f(\kappa) = \det(I-L{(\kappa)})$ is analytic with respect to $\kappa$
for $\kappa \in \mathbb{C} \backslash \{0\}$.
\end{lemma}
\begin{proof}
For $L \in C[0, 2\pi]^{2}$, $\det(I-L) = \sum_{m=0}^{\infty}\frac{(-1)^{m}}{m!} \int_{0}^{2\pi}...\int_{0}^{2\pi} \det(L(t_{p},t_{q})_{p,q=1}^{m})dt_{1}...dt_{m}$ \cite[page 112]{gohberg}. $L(t_{p},t_{q})$ is analytic in $\kappa$ on $\mathbb{C} \backslash \{0\}$ by construction. Define $L_{m}:=\det(L(t_{p},t_{q})_{p,q=1}^{m})$ then $L_{m}$ is analytic in $\kappa$ on $\mathbb{C} \backslash \{0\}$. The idea is to show that $\det(I-L)$ is the uniform limit of the sequence of analytic functions $\{L_{m}\}$ on any compact set in
$\mathbb{C} \backslash \{0\}$. 
\begin{align*}
R_{M} &:=  \biggl|\sum_{m=M}^{\infty} \frac{(-1)^{m}}{m!} \int_{[0, 2\pi]^{m} }L_{m}(t_{1},....,t_{m})dt_{1}...dt_{m} \biggr| \\
            & \leq \sum_{m=M}^{\infty} \frac{1}{m!} (2\pi)^{m} \Vert L_{m} \Vert_{L^{\infty}} \\
            & \leq \sum_{m=M}^{\infty} \frac{1}{m!} m^{\frac{m}{2}} (2\pi \Vert L \Vert_{L^{\infty}})^{m}
 \end{align*}
The second inequality comes from Hadamard's Inequality.
As proved in \cite{bornemann}, the power series $\Phi(z) = \sum_{m=1}^{\infty} \frac{m^{(m+2)/2}}{m!}z^{m}$ defines an entire function on $\mathbb{C}$, together with the fact that $\Vert L \Vert_{L_{\infty}}$ is uniformly continuous in $\kappa$ on any compact set in $\mathbb{C} \backslash \{0\}$, we have $R_{M} \rightarrow 0$ as $M \rightarrow \infty$ locally uniformly in $\kappa$ on $\mathbb{C} \backslash \{0\}$. Thus $\det(I-L)$ is the locally uniformly convergent limit of a sequence of analytic functions in $\kappa$ on $\mathbb{C} \backslash \{0\}$. The claim follows.
\end{proof}

An analogous statement holds for our numerical approximation, namely
that $f_N(\kappa)$ is analytic in $\kappa$ close enough
to the positive real axis.
This follows from Lemma \eqref{fnm}, which says $f_{N}(\kappa) = \det(I-M_N(\kappa))$, an $N$-dimensional matrix determinant, and the fact that matrix entries are linear combinations of Hankel functions.
From Theorem~\ref{t:main}, $f_N$ vanishes exponentially fast at
each eigefrequency $\kappa_j$,
and thus, if we assume that the derivative $f'_N(\kappa_j)$ is bounded away from
zero for sufficiently large $N$,
the roots of $f_N$ approach the true eigenfrequencies with accuracy
exponential in $N$.

\remark{We do not prove that $f_N(\kappa)$ converges to $f(\kappa)$
exponentially for all $\kappa$;
indeed the numerical evidence (Section~\ref{s:num}) will be
that this convergence is merely algebraic for $\kappa$
away from eigenfrequencies.}

All that is now needed is an efficient method to find good approximations to the
real roots of the numerical Fredholm determinant $f_N(\kappa)$.
We propose Boyd's ``degree-doubling'' method \cite{boyd},
which, given that our function is analytic on the real axis,
is spectrally accurate in the number of function evaluations \cite{boyd}.
Thus just a few evaluations per root found will be enough
to approach machine accuracy.

Say we wish to find roots of $f_N$ in an interval $\kappa\in[a,b]$.
We change variable to
$\kappa(\theta) = \frac{b+a}{2} + \frac{b-a}{2} \cos \theta$,
choose a small number $M$, and evaluate the function on
a regular grid in $\theta$, i.e.\ $f_j = f_N(\kappa(\pi j/M))$,  $j=1,\ldots,2M$.
Note that only $M+1$ evaluations are needed since $\kappa(2\pi-\theta)
=\kappa(\theta)$.
Since $f_N(\kappa(\theta))$ is a $2\pi$-periodic function of $\theta$
analytic in a neighborhood of the real axis,
the Fourier representation
\be
f_N(\kappa(\theta)) \approx \sum_{m=-M}^M c_m e^{im\theta}
\label{fou}
\ee
is exponentially convergent in $M$.
(This is equivalent to a Chebyshev expansion in the variable $\kappa$.)
The coefficients $\{c_m\}$ are
computed via the fast Fourier transform of the vector $\{f_j\}$.
In practice we start with $M=4$, and double $M$, reusing previous
$f_j$ values, until $|c_M/c_0|\le 10^{-12}$.
Writing $z=e^{i\theta}$, \eqref{fou} is a Laurent expansion in $z$,
hence
$$
q(z) \;:=\;
z^M \!\! \sum_{m=-M}^M c_m z^m
$$
is a degree-$2M$ Taylor series with the same nonzero roots.
These roots are found by insertion of the vector $\{c_m\}$
into a companion matrix \cite{companion}
and finding its eigenvalues $\mu_i$ at a cost of $O(M^3)$
(although we note that evaluation of $f_j$ dominates over this
cost by far).
Finally, only the eigenvalues $\mu_i$ within $\epsilon$ of the unit
circle are kept; these are converted back to give the roots
$\kappa_i = \frac{b+a}{2} + \frac{b-a}{2} \re \mu_i$.
The imaginary parts
\be
\label{eq: betai}
\beta_i := \frac{b-a}{2} \im \mu_{i}
\ee
we observe are good indicators of of the size of errors in the roots.
This algorithm is available in \mpspack\ \cite{mpspack}
as \verb?@utils/intervalrootsboyd.m?

Finally, if the above criterion for Fourier series decay is not
met with $M=512$, or if it turns out that $|\beta_i|>\beta$, where $\beta$
is a fixed algorithm parameter,
then the interval $[a,b]$ is instead subdivided
and the process repeated on the smaller intervals.

\section{Numerical results for a simply-connected domain}
\label{s:num}
\subsection{Convergence of the Fredholm determinant}
To demonstrate the convergence of $f_{N}(\kappa)$ given by \eqref{fnk}
as a function of $N$, the number of quadrature nodes on $\Gamma$,
we use the non-symmetric planar domain described in Fig.~\ref{fig: ns}.
We test $\kappa$ values near the 100th eigenfrequency $\kappa_{100}$.
As the graph in Fig.~\ref{fig: ns} shows, for $\kappa=\kappa_{100}$,
convergence to zero is at least exponential.
However, as $\kappa$ moves away from the eigenfrequency,
the colorscale plot shows that the initial exponential
convergence deteriorates to much slower algebraic convergence.
We believe the latter is of third order,
although we do not have a proof of this.
(A possible explanation for third-order convergence is that it is
what a naive Nystr\"om method without Kress' analytic split would
give for the operator $I-2D$.)

\begin{figure}[!ht] 
(a) \raisebox{-1.6in}{\includegraphics[width=0.25\textwidth]{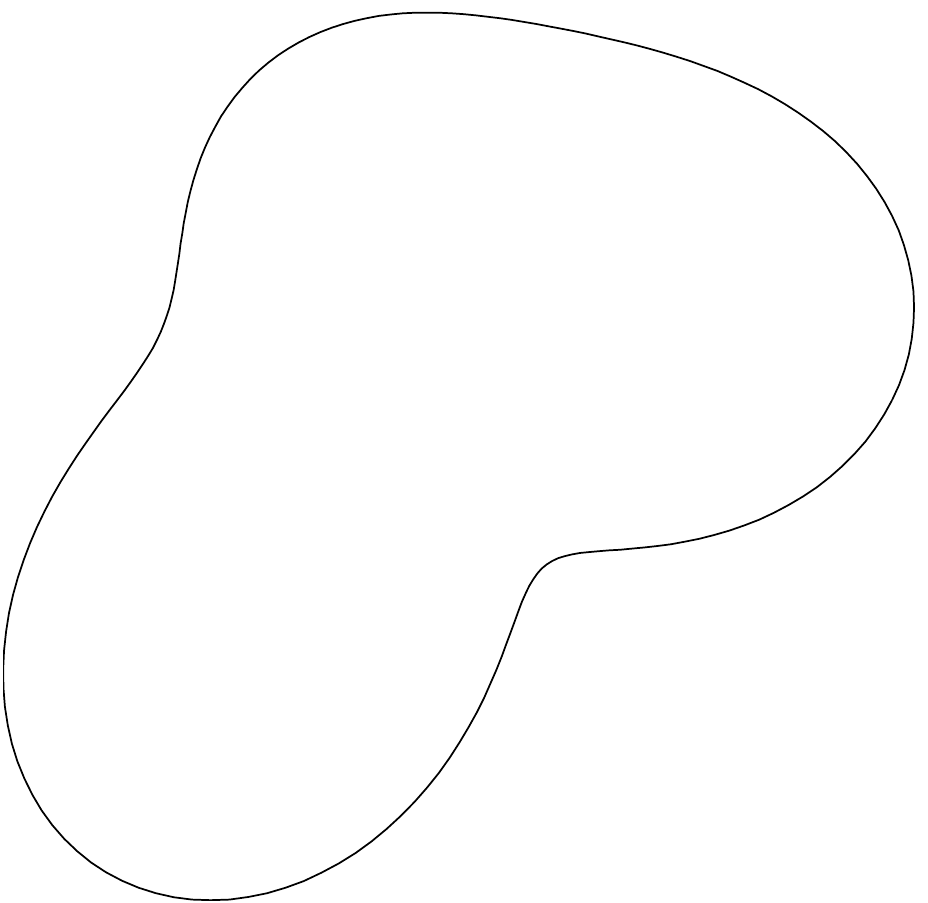}}
\hfill
(b) \raisebox{-1.8in}{\includegraphics[width=0.45\textwidth]{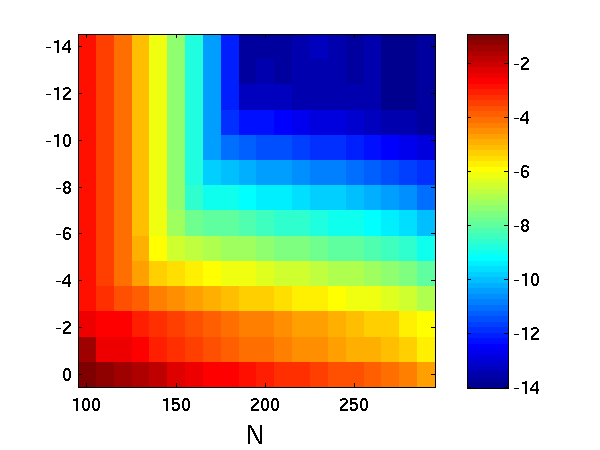}}\\
\centering 
(c) \raisebox{-1.3in}{\includegraphics[width=0.95\textwidth]{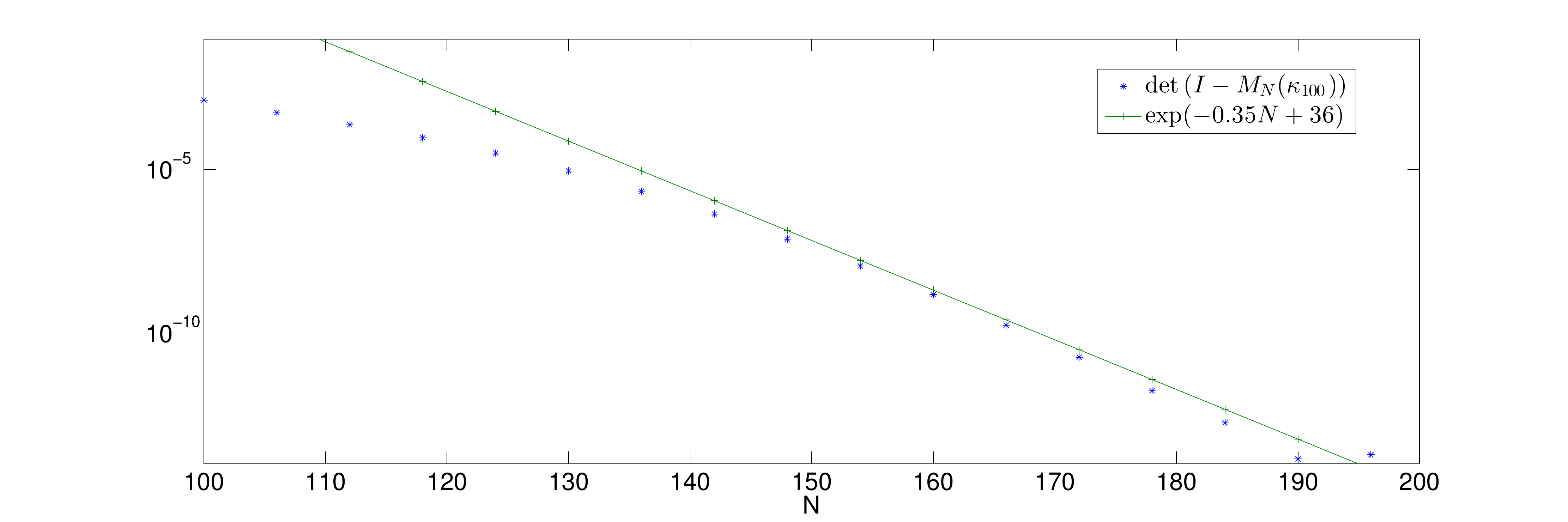}}
   \caption{
   (a) Domain defined by $r(\theta) = 1+0.2\cos{3\theta}+0.3\sin{2\theta}$. 
   (b) $\log_{10}{f_{N}(\kappa)}$ near $\kappa_{100} = 20.43009417604$ (converged value).
           the vertical axis shows $\log_{10}{(\kappa-\kappa_{100})}$; 
    (c) the convergence of $f_{N}(\kappa_{100})$ to zero. $N$ is the number of quadrature nodes on the boundary.} 
         \label{fig: ns} 
\end{figure} 

\subsection{Convergence of the determinant roots to the eigenfrequencies}
With the same domain as above, we now verify the claim of the previous
section that a root converges as fast as the rate of vanishing
of the determinant at a true eigenfrequency.
We solve for roots of $f_{N}(\kappa)$ on the interval $[20.4, 20.5]$
containing $\kappa_{100}$ using the method of Section \ref{s:boyd}.
Fig.~\ref{fig: zeroConvNS}
shows at least exponential convergence of the numerical root
to its converged value $\kappa_{100}$.
Note that 14-digit accuracy (15-digit relative accuracy)
is achieved using only $N=180$.

\begin{figure}[!ht] 
   \includegraphics[width=0.9\textwidth]{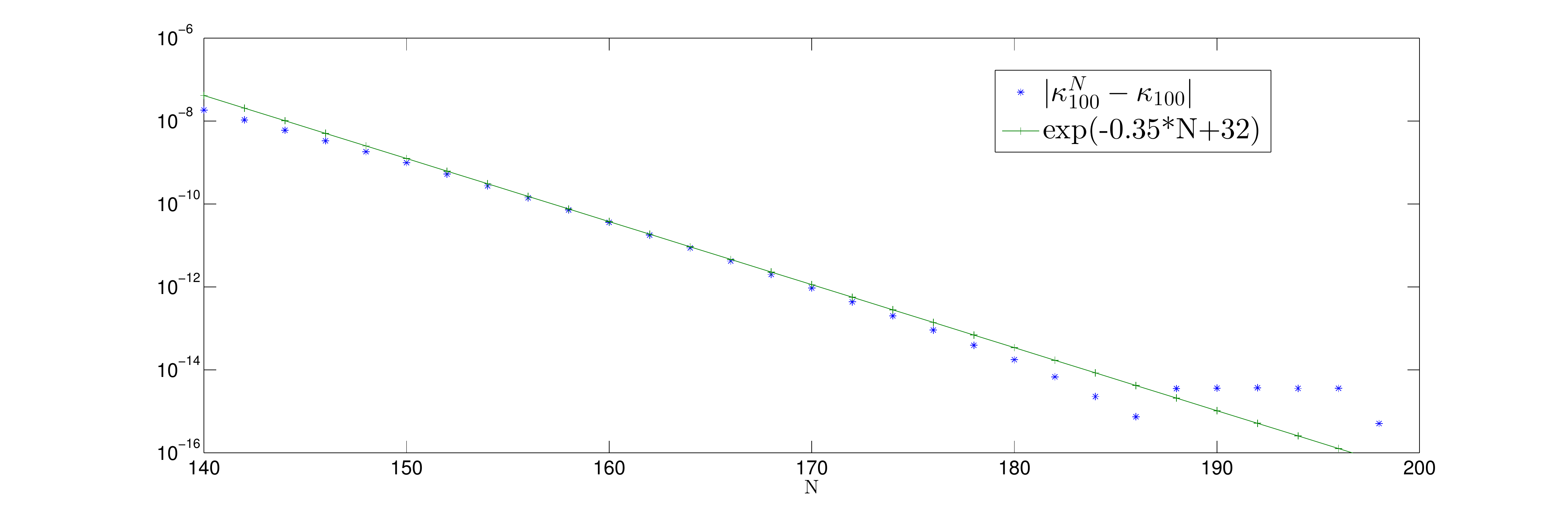}
   \caption{Convergence of the eigenfrequency error with $N$.
The vertical axis shows the error
(relative to its converged value)
of the root found by the method of Sec.~\ref{s:boyd} at each $N$. 
$N$ is the number of quadrature nodes on the boundary.
}
\label{fig: zeroConvNS}
\end{figure}

\begin{figure}[!ht]
(a) \raisebox{-2.4in}{\includegraphics[width=0.45\textwidth]{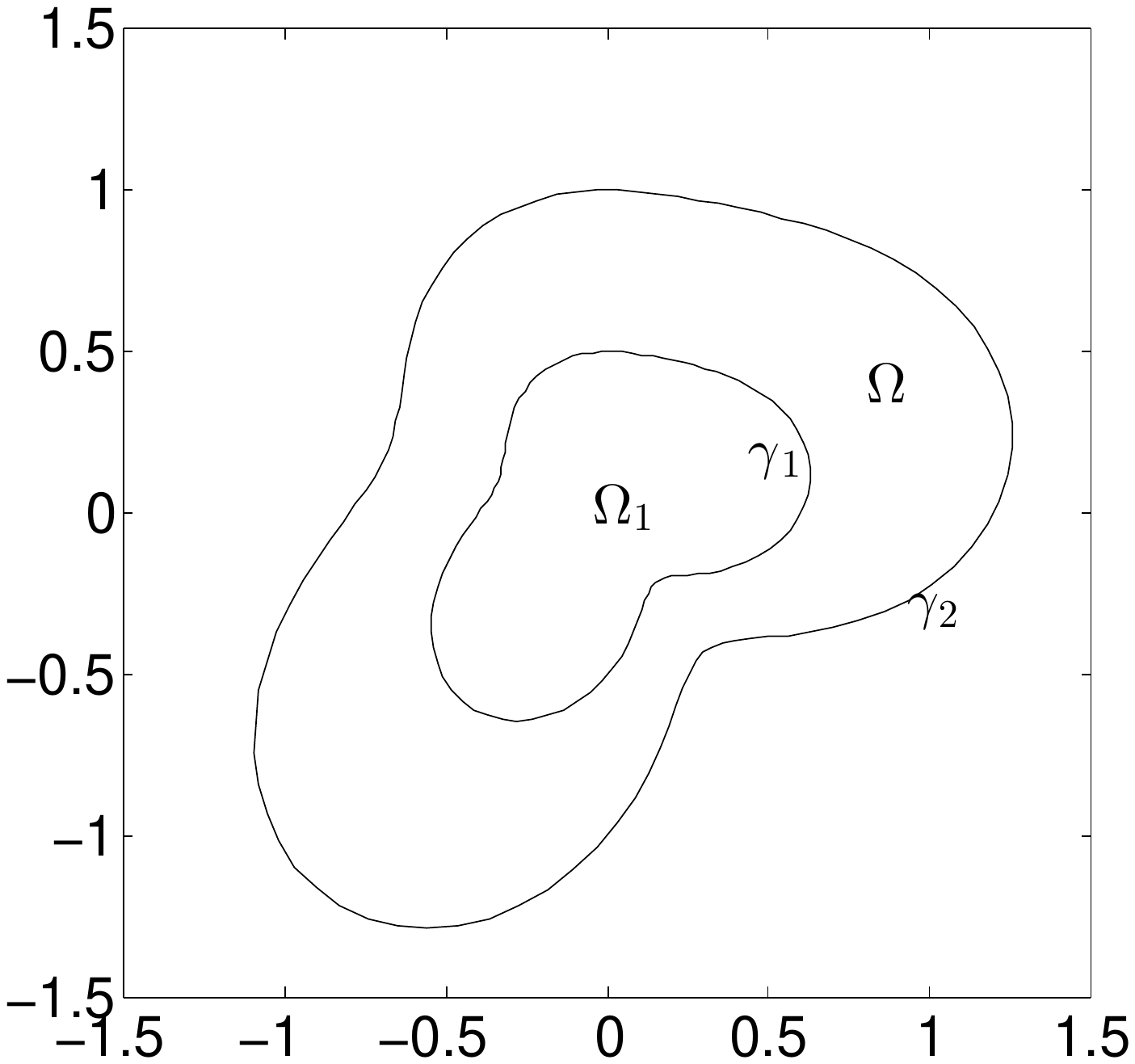}}
\hfill
(b) \raisebox{-2.4in}{\includegraphics[width=0.45\textwidth]{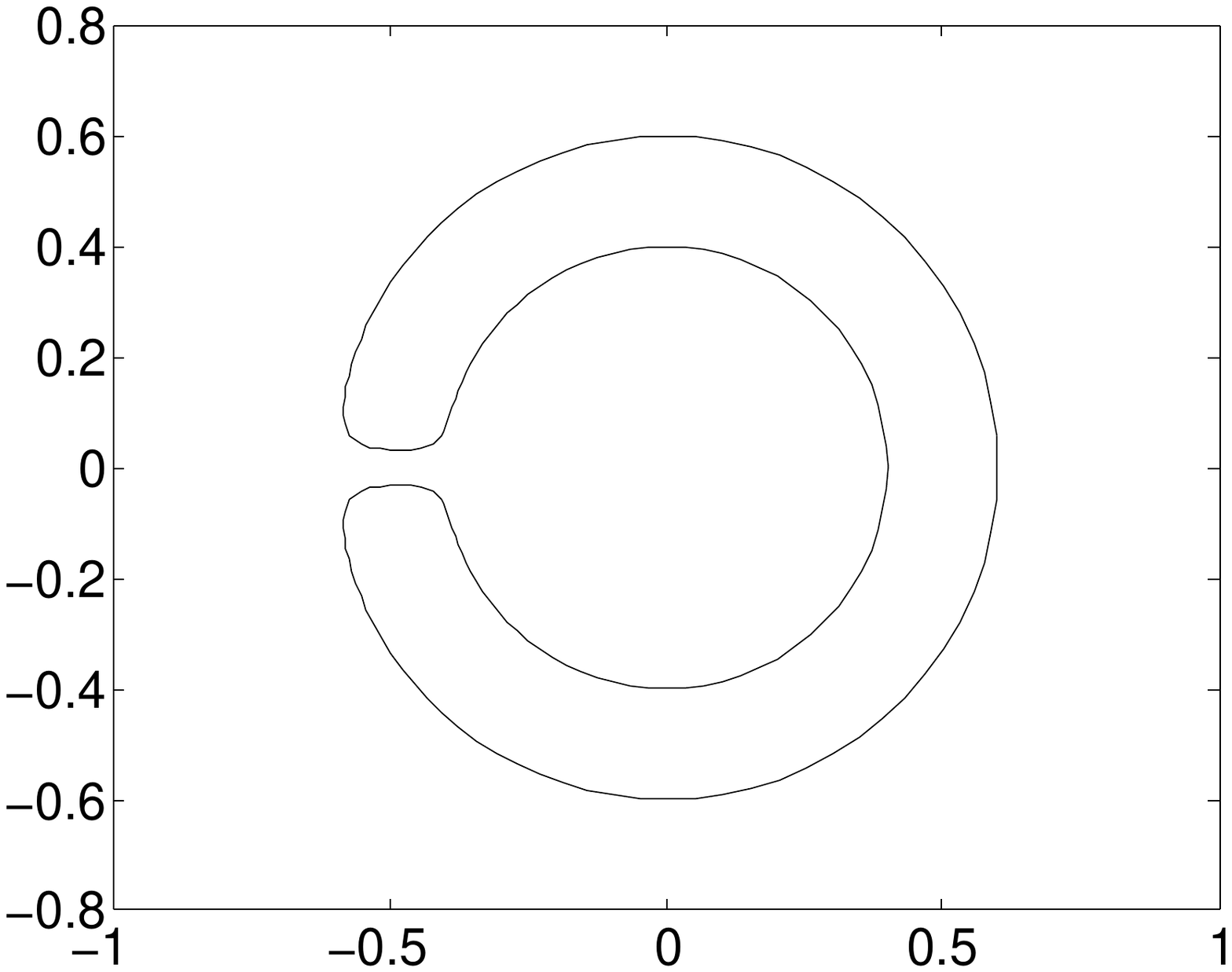}}
\vspace{-0.7in}
\caption{
(a) Annular domain with boundary curves $\gamma_{2}: r(\theta) = 1+0.2\cos{3\theta}+0.3\sin{2\theta}$, and $\gamma_{1}: r(\theta) = 0.5+0.1\cos{3\theta}+0.15\sin{2\theta}$, $0 \leq \theta \leq 2\pi$.
(b) Crescent-shaped domain with strong exterior resonances,
with polar parametric description
$r(s) = \frac{0.2}{1+\exp{(4(s-3\pi/2)(s-\pi/2)})}+0.4$,
$\theta(s) = -\frac{49}{50}\pi\sin{s}$,
$0 \leq s \leq 2\pi$.
}
\label{fig:annularSector}
\end{figure}

\section{The resonance phenomenon and multiply-connected domains}
\label{s:res}
If the domain $\Omega$ has a hole,
Theorem~\ref{kress} does not apply,
and we cannot therefore know that every root of the Fredholm determinant
$f(\kappa)$ indicates a
Dirichlet eigenfrequency of $\Omega$.
The following lemma characterizes this new scenario.
We denote the inner boundary $\gamma_{1}$ and outer boundary $\gamma_{2}$.
Also let $\Omega_{1}$ be the domain that $\gamma_{1}$ encloses;
see Fig.~\ref{fig:annularSector}(a).
\begin{lemma}
Let $\Omega$ be a domain with a hole $\Omega_{1}$, and boundary
$\Gamma = \gamma_{1} \cup \gamma_{2}$.
Then the operator $I-2D$ on $\Gamma$ has a nontrivial nullspace if $\kappa$ is a Neumann eigenfrequency of $\Omega_{1}$.
\label{l:hole}
\end{lemma}
Recall that Neumann eigenfrequencies are the discrete $\kappa$ values where
nontrivial solutions to \eqref{helm} with $u_n=0$ on $\Gamma$ exist.
Such eigenfrequencies generally do not coincide with the desired Dirichlet
eigenfrequencies, thus our method of double layer potential produces incorrect roots for domains not simply connected.
The obvious generalization of the lemma
to domains with multiple holes also holds.

\begin{proof}
$\Omega_{1}$ has countably many interior Neumann eigenmodes. For any such eigenmode with boundary data $u$ and $u_{n}\equiv 0$ on $\gamma_{1}$,
we can first extend $u$ to $\tilde{u}$ defined on $\gamma_{1} \cup \gamma_{2}$ by setting $\tilde{u}=0$ on $\gamma_{2}$. We construct the double-layer potential
$\mu:=\mathcal{D} \tilde{u}$ at the corresponding eigenfrequency. Thus $\mu$ is a solution to the Helmholtz equation on $\mathbb{R}^2 \backslash \gamma_{1}$. Furthermore, for $x \in \mathbb{R}^{2} \backslash \overline {\Omega_{1}}$,
$\mu (x) = \mathcal{D} \tilde{u}
= \mathcal{D} u = \mathcal{D}u - \mathcal{S} u_{n} = 0$
by Green's representation theorem \eqref{GRF}
applied to the exterior of $\Omega_1$.
Consider the continuous extension of $\mu$ from inside $\Omega$ to $\gamma_{1}$, from the jump relation \eqref{JR3},
we see $(D-\frac{1}{2})\tilde{u} = 0$, i.e. the integral equation has a nontrivial solution.
\end{proof}

For such a domain, if one solves for the roots of $f_{N}(\kappa)$, one gets not only the Dirichlet eigenfrequencies of $\Omega$, but also the Neumann eigenfreqencies of the enclosed domain $\Omega_{1}$, which we call the spurious roots.

This has an important consequence:
even for a simply connected domain, as the geometry becomes more concave,
spurious roots may show up numerically
(first observed in this context by B\"acker \cite[Sec.~3.3.6]{backerbim}).
The operator $I-2D(\kappa)$ becomes singular for a $\kappa$ very close
to the real axis, resulting in a determinant very close to zero for a real
$\kappa$.
Any root-finding method working in finite precision thus cannot distinguish those $\kappa$ from true eigenfrequencies. Physically, this corresponds to a resonance of the exterior Neumann boundary-value problem for the domain's boundary,
since the operator $I-2D$ also arises in the potential-theoretic solution of this problem.
In \cite{ellipseres} it is proved, via an elliptical cavity domain,
that such boundary value problem resonances may exist with $\im \kappa$
becoming exponentially small as $\re \kappa$ grows.

We now demonstrate this problem, using the concave domain of Fig.~\ref{fig:annularSector}(b). It closely resembles, and can be viewed as a smooth approximation of, an annular sector with inner radius $0.4$, outer radius $0.6 $ and angular ``openness'' parameter $\frac{49}{50}\pi$. The disk with radius $0.4$ has a Neumann eigenfrequency $\kappa_{N} = 26.2996521844$. And indeed, for the
cresent domain, our
root-finding method returns a spurious root $\kappa_{0} = 26.30048303974$,
clearly visible in Fig.~\ref{fig: sweep}(b). This $\kappa_{0}$ is not exactly $\kappa_{N}$ because the crescent domain is not an exact annulus.

\subsection{A new representation for the Dirichlet eigenvalue problem}

We can remedy the above non-robustness by constructing the boundary integral equation using the {\em combined field} potential,
$$u : = \mathcal{D}\varphi+i \eta \mathcal{S}\varphi~,$$
where $\eta$ is a real parameter which, following \cite{coltonkress}, we
set to be $\kappa$.
This is standard in the acoustic scattering literature, but to our
knowledge has not been used for the eigenvalue problem before.
(The idea was suggested in one sentence of \cite[Sec.~3.3.6]{backerbim}.)
Enforcing the Dirchlet boundary condition \eqref{bc} on the combined field potential gives the CFIE
\begin{equation}
\label{CFIE}
(I-2D-2i\eta S)\varphi= 0
\end{equation}
For the CFIE we have the following equivalence relation;
in contrast to Theorem~\ref{kress}, it does not require simply connectedness
of the domain.
\begin{theorem}  
\label{thm: CFIE}
Let $\Omega$ be a bounded domain with twice continuously differentiable boundary $\Gamma$.
For each $\kappa \in \mathbb{C}\backslash\{0\}$ with $\operatorname{Im}\kappa \geq 0$,
$\kappa^2$ is a Dirichlet eigenvalue of $\Omega$ if and only if
$I-2D(\kappa)-2i\eta S(\kappa)$ has a nontrivial nullspace,
where $\eta \neq 0$ is an arbitrary real number with
$\eta \operatorname{Re} \kappa \geq 0$.
\end{theorem}     
\begin{proof} 
"$\Rightarrow$" Suppose $u$ is an eigenfunction,
using the same argument as in Lemma \ref{l:nullspace} we have 
$(1-2D^T)u_{n}^{-}=0$. Green's representation theorem \ref{GRF} says
$Su_{n}^{-}=0$. Thus $(1-2D^T-2i\eta S)u_{n}^{-}=0$
In the dual system 
$\langle C(\Gamma), C(\Gamma) \rangle$ with the bilinear form
$\langle \varphi, \psi \rangle:= \int_{\Gamma} \varphi(x) \psi(x) dx$,
$S$ is self-adjoint and the adjoint of $D$ is $D^T$ \cite[p. 41]{LIE}.
By the Fredholm alternative, $I-2D-2i\eta S$ has a nontrivial nullspace.

"$\Leftarrow$" Suppose $\varphi\in \Null(I-2D-2i\eta S)$ and $\varphi$ is not identically zero. 
Consider $\mu:=(\mathcal{D}+i\eta \mathcal{S})\varphi\in C^2(\mathbb{R}^2\backslash\Gamma)$, 
then $\mu$ satisfies \eqref{helm} by construction. 
We look at $\mu^{\pm}$ and $\mu_{n}^{\pm}$ using the jump relations \eqref{JR1} through \eqref{JR4}.
First, $\mu$ satisfies the zero Dirichlet boundary condition for the interior problem since
$\mu^{-}=(D-\frac{1}{2}+i\eta S)\varphi =0$.
So now we need only show that $\mu$ is nontrivial.
Suppose $\mu$ is identically zero in $\Omega$, then
$\mu_{n}^{-}=[T+i\eta (D^T+\frac{1}{2})]\varphi=0$.
Thus we have
$\mu^{+} = (D+\frac{1}{2} +i\eta S)\varphi = \varphi$, and
$\mu_{n}^{+} = [T+i\eta(D^T-\frac{1}{2})]\varphi = -i\eta \varphi$.
Therefore $\mu$ is a solution to \eqref{helm} on $\Omega_{+}$ with the impedance boundary condition
\be
i\eta \mu+\mu_{n}\;=\;0 \qquad \mbox{ on } \hspace{0.5cm} \Gamma
~,
\label{imp}
\ee
and $\mu$ is radiative in
the exterior component containing infinity.
In this infinite component $\mu$ has a unique solution when $\eta \operatorname{Re} \kappa \geq 0$ ~\cite[p.~97]{CK83},
thus $\mu\equiv 0$ in this component.
So $\varphi$ must be identically zero on the
boundary of this component. If $\Omega$ has no holes, we have reached a contradiction.
Otherwise, let $\Omega_{1}$ be any of the holes in $\Omega$, with boundary $\gamma_1$.
Let $n'$ be the unit normal vector pointing to the exterior of $\Omega_{1}$, then
$n' = -n|_{\gamma_{1}}$. $\mu$ is a solution to \eqref{helm} on $\Omega_{1}$
with boundary condition
\be
i\eta \mu-\mu_{n'}\;=\;0 \qquad \mbox{ on } \hspace{0.5cm} \gamma_{1}
~,
\label{imp1}
\ee
Multiplying each side of \eqref{helm} by $\overline{\mu}$, 
integrating over $\Omega_{1}$ and applying Green's first identity
and \eqref{imp1}, we get
\begin{equation}
\kappa^{2} \| \mu \|^{2}_{L^{2}(\Omega_{1})} = \| \nabla \mu \|^{2}_{L^{2}(\Omega_{1})}-i \eta \| \mu \|^{2}_{L^{2}(\gamma_{1})}
~.
\end{equation}
Taking the imaginary we have
$2\operatorname{Re} \kappa \operatorname{Im}\kappa \| \mu \|^{2}_{L^{2}(\Omega_{1})}  = -\eta \| \mu \|^{2}_{L^{2}(\gamma_{1})}$, which is impossible given all the conditions on $\kappa$ and $\eta$
 unless $\mu$ vanishes on $\gamma_{1}$.
Hence $\mu_{n'}$ vanishes on $\gamma_1$ by \eqref{imp1}. By Green's representation theorem,
$\mu$ is identically zero in $\Omega_{1}$.
We have shown $\mu\equiv0$ in all of $\mathbb{R}^2\backslash\overline{\Omega}$
and this means $\varphi$ is identically zero on $\Gamma$,
which is a contradiction.
So $\mu$ is a nontrivial solution to \eqref{helm} -\eqref{bc}, hence an eigenfunction.
\end{proof}

Thus by adopting the combined field integral equation, we
have a robust method with no spurious frequencies where the
boundary operator is singular.
We show this in Fig.~\ref{fig: sweep},
where we show the minimum singular value of (the Nystr\"om approximation to)
the original operator and of the CFIE,
for (a) a doubly-connected domain and (b) a simply-connected domain with
strong exterior resonances.
In both cases this shows that the CFIE removes the spurious roots.

We now mention numerical implementation issues for the
CFIE formulation.

For the spectrally-accurate discretization of
the single-layer operator, we use the same method
as for the double-layer, replacing
$L$ by $Q(t,s) = \Phi(x(t),x(s)) |x'(s)|$,
and replacing the logarithmically singular term
\eqref{L1} by \cite[Eq.~(2.6)]{kress91},
\be
Q^{(1)}(t,s):=-\frac{1}{2\pi}
J_0(\kappa r(t,s)) \sqrt{x_1'(s)^2+x_2'(s)^2}
\label{Q1}
\ee
and defining $Q^{(2)}$ as before by the difference \eqref{L2}.
The resulting matrix we call $Q_{N}{(\kappa)}$.
For each $N$ the determinant of the $N$-node Nystr\"om discretization matrix
$I-M_{N}(\kappa)-i\eta Q_{N}(\kappa)$ is analytic in $\kappa$. Thus we are able to apply the same
root-finding method to it as before, and achieve rapid convergence with $N$
for the roots, hence eigenvalues found.
\begin{remark}
Note that for $\eta \neq 0$, $D+i\eta S$ is no longer in $\mathcal{J}_{1}(L^{2}[0, 2\pi])$,
so the main convergence theorem \ref{t:main} does not readily apply. 
Instead let $\mathcal{J}_{2}(L^{2}[0, 2\pi])$ be the space of {\em Hilbert-Schmidt operators} on $L^{2}[0, 2\pi]$, 
which is the collection of all linear operators with square summable singular values,
 then $D+i\eta S$ is in $\mathcal{J}_{2}(L^{2}[0, 2\pi])$. 
Given $A \in  \mathcal{J}_{2}(L^{2}[0, 2\pi])$, $ \displaystyle \prod_{j=1}^\infty(I-\lambda_{j}(A))$
 is not necessarily convergent. However, we expect that numerically,
  the convergence theorem \ref{t:main} should be {\em close} to holding.
  Since the singular values of $S$ decay like $\frac{1}{j}$,
  their sum only diverges logarithmically.
 In addition, our experiments show that $\det(I-M_{N}(\kappa)-i\eta Q_{N}(\kappa))$
  converges to zero as $N \rightarrow \infty$ if and only if $\kappa=\kappa_{j}$.
 \end{remark}
 
\begin{figure}[!ht]
(a)\raisebox{-1.8in}{\includegraphics[width=0.95\textwidth]{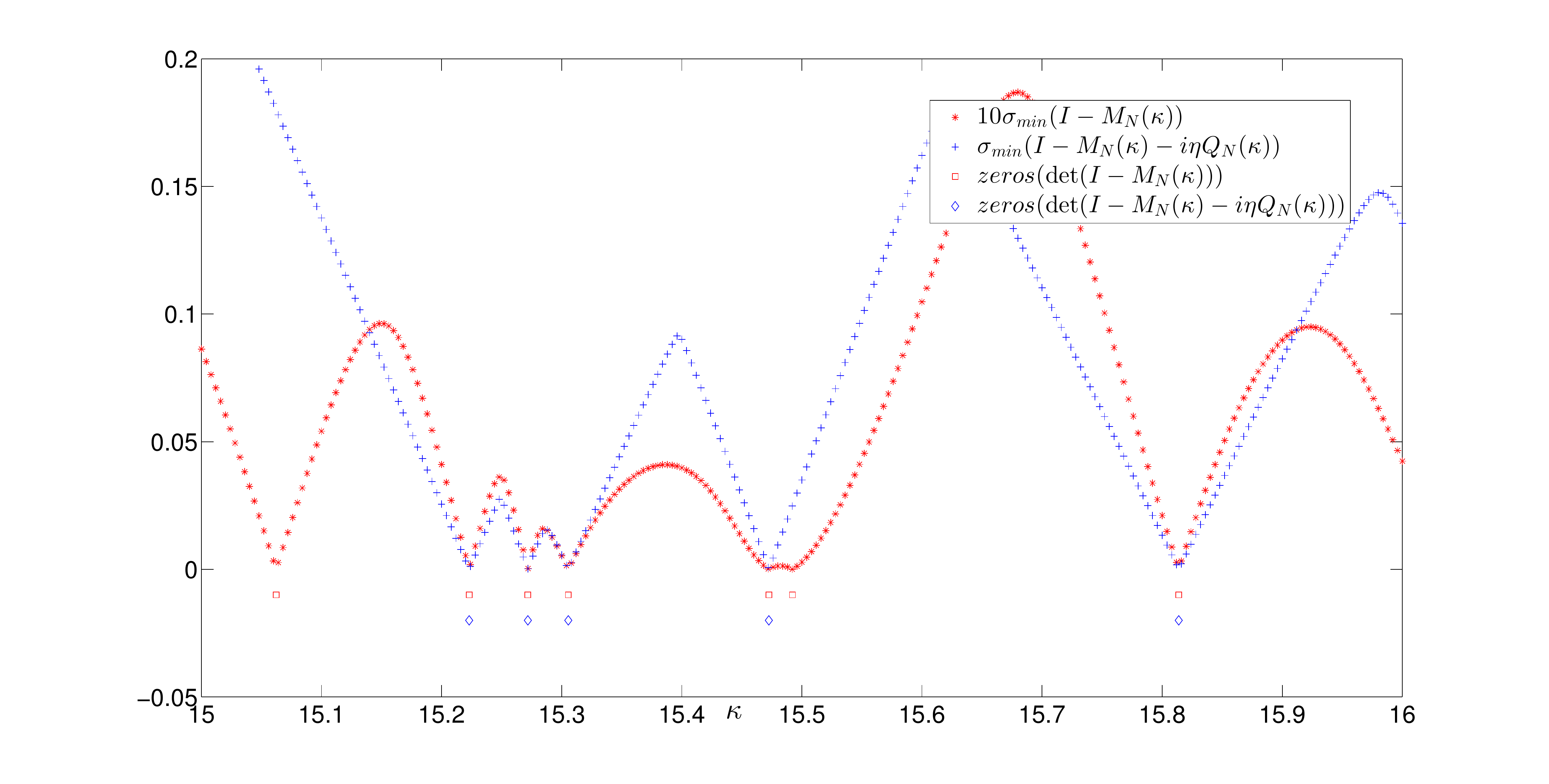}}
\\
(b)\raisebox{-1.8in}{\includegraphics[width=0.95\textwidth]{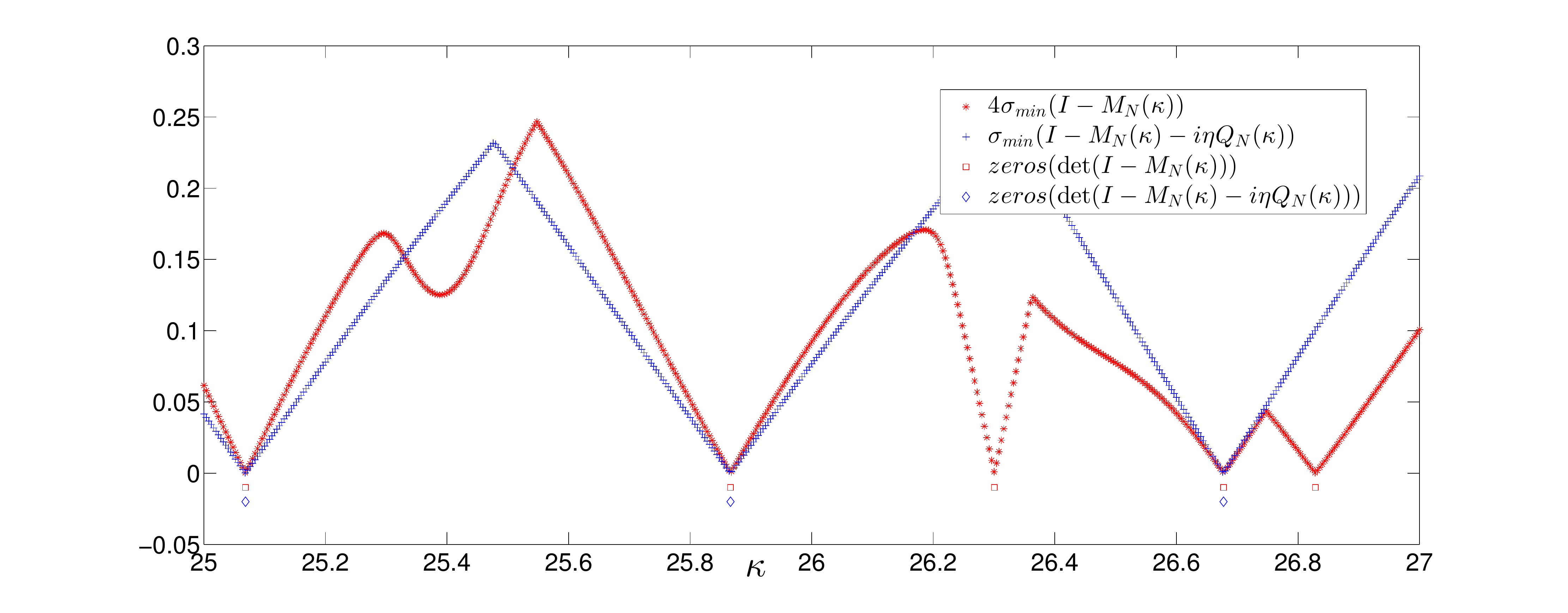}}
\caption{
Lowest singular values vs frequency $\kappa$, for:
(a) the annular domain Fig.~\ref{fig:annularSector}(a);
(b) the crescent domain Fig.~\ref{fig:annularSector}(b).
$\sigma_{min} $ denotes the smallest singular value
of the discretized
$I-2D(\kappa)$ (red), and CFIE $I-2D(\kappa)-2i\eta S(\kappa)$ (blue).
The true eigenfrequencies are shown by the blue diamonds.
$N$ is the number of quadrature nodes on the boundary}
\label{fig: sweep}
\end{figure}

\section{Close eigenfrequencies and the singular value method}
\label{s:svd}
Our root-finding method worsens in accuracy when
$f(\kappa)$ has close roots, or roots with multiplicity higher than one.
\footnote{Note that we do not expect this to occur too often, since
for a generic domain eigenvalues are all simple \cite{KS}.}
In this section we discuss how we overcome this problem if it does occur,
by reverting to the standard SVD method.
Indeed, no method that relies on evaluating the Fredholm determinant
$f(\kappa)$ alone could succeed in this case, because the root-finding
problem is well known to be ill-conditioned with respect to perturbations in
the function
(eg, for a polynomial, perturbations in its coefficients).

We discuss the case of two close eigenfrequencies
$\kappa_j \approx \kappa_{j+1}$.
Then $f(\kappa)=t(\kappa)(\kappa-\kappa_j)(\kappa-\kappa_{j+1})$
for some locally smooth function $t(\kappa)$.
For simplicity, let $f$ be perturbed by a constant value $\eps$;
then, for small $\eps$, the change induced in the root $\kappa_j$
is of size
\be
\delta\kappa \approx \biggl|\frac{\eps}{f'(\kappa_j)}\biggr| =
\biggl|\frac{\eps}{(\kappa_{j+1}-\kappa_j)t(\kappa_j)}\biggr|
~,
\label{cond}
\ee
which blows up inversely with the gap between the eigenfrequencies.
This particular perturbation demonstrates the ill-conditioning; other perturbations lead generically to a similar effect.
Even for $\eps \approx 10^{-16}$ we may only retain accuracy $O(\eps^{1/2})$
as two roots approach each other, and more if there are more close
roots or a higher-order degeneracy.

To remedy this, when two roots are found closer than $s \approx 10^{m} \eps$,
where $m$ is the desired number of digits of accuracy in rootfinding,
we propose switching to a more expensive method based on the SVD.
This requires finding the lowest singular values of the CFIE Nystr\"om
matrix $I-M_N(\kappa) - i\eta Q_N(\kappa)$,
and is very similar to existing eigenvalue solvers
\cite{backerbim,gsvd}.
We only use the SVD when forced to do so
since, due to the high cost of the SVD, and the increased number of
function evaluations required to find each root,
we will show that it is
an order of magnitude less efficient than our proposed method.

Thus the choice of the parameter $s$ affects the robustness and the speed of the algorithm. The smaller it is, the less often roots less than $s$ apart will
occur, and thus the faster the computation.
However, smaller $s$ causes a worsening of the accuracy of close roots.
This is more severe for multiple roots:
for $n>1$, an order-$n$ root has error on the order of $\eps^{\frac{1}{n}}$,
Thus to obtain desired accuracy, $s$ has to be set to be large enough.
In practice we fix $s = 10^{-3}$.

Once we switch to using the SVD on an interval of frequency $\kappa$,
the smallest singular value $\sigma_{\min}(I-M_{N}(\kappa)-i\eta Q_{N}(\kappa))$ is far from analytic in $\kappa$ (see Fig.~\ref{fig: sweep} which shows the typical W-shaped function), so the Boyd's method is not useful.
Instead we use recursive subdivision starting on a regular grid of values, followed by iterative parabolic fitting of $\sigma_{\min}^2(I-M_{N}(\kappa)-i\eta Q_{N}(\kappa))$
as detailed in \cite[Appendix~B]{sca}. This algorithm is available in \mpspack\ \cite{mpspack}
as \verb?@evp/gridminfit.m?

To demonstrate the higher accuracy of the SVD method over the Boyd's method in the presence of close eigenfrequencies, we choose an ellipse domain,
and vary its eccentricity to cause a near-degeneracy of controllable separation $\kappa_{j+1}-\kappa_j$.
Fig.~\ref{fig:ellipse} shows the eigenfrequencies passing through
each other as a function of the eccentricity,
solved by the determinant (red) and by the SVD methods (blue).
Errors of absolute size around $10^{-7}$ appear in the determinant
method but not the SVD method.
As expected from \eqref{cond}, we see the errors $\delta\kappa$ blow up like $\frac{1}{|\kappa_{j+1}-\kappa_{j}|}$.
\begin{figure}[h!]
  \includegraphics[width=1\textwidth]{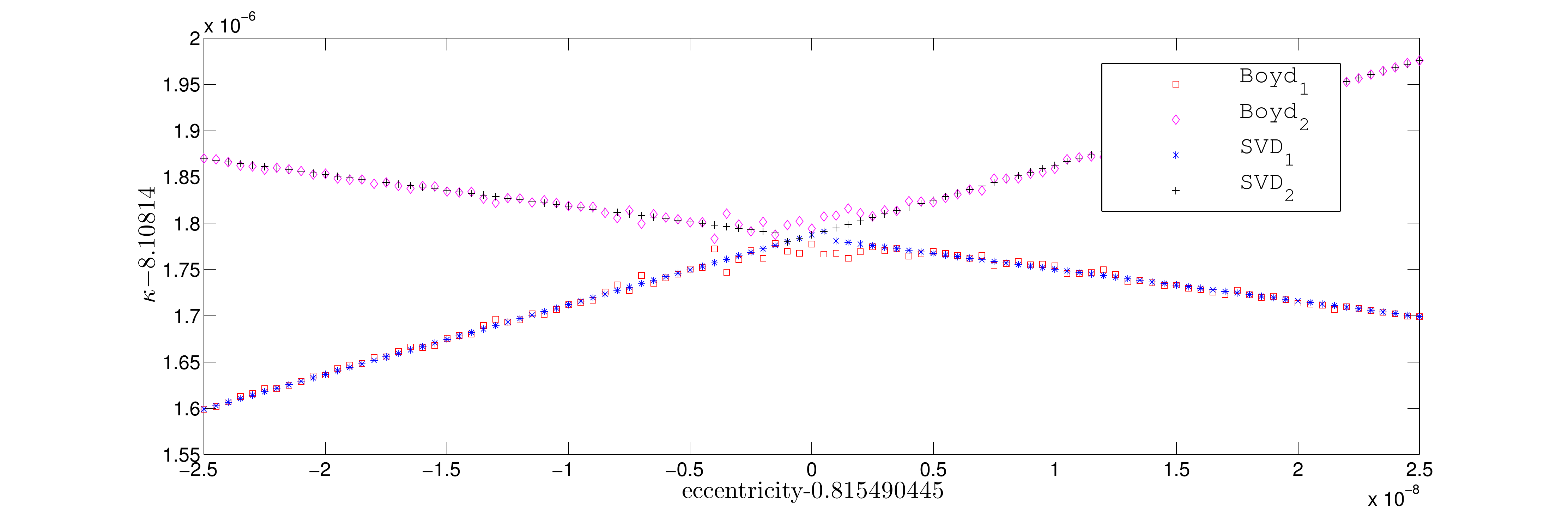}
  \caption{Two close eigenfrequencies of an ellipse crossing as a function of
the eccentricity.
Red shows values computed by Boyd's method applied to the determinant on the frequency interval $\kappa\in[7, 9]$. Blue shows values computed by the SVD method
of Sec.~\ref{s:svd}. $\eta$ is set to be zero since we expect and observe
no exterior resonances.}
  \label{fig:ellipse}
\end{figure}

\section{Numerical performance of the solver}
\label{s:ps}
In this section we demonstrate the
improved efficiency of our solver, the Boyd's method with determinant, compared
to an existing boundary-integral solver, namely the 
SVD method described in the previous section.
We used a Linux workstation with two quad-core E5-2643 3.3GHz Xeon CPUs,
running MATLAB R2013b,
except for Hankel function evaluations which use Rokhlin's
fortran code {\tt hank103.f} (eg see \cite{mpspack}).

\subsection{Non-resonant domain solved via pure double-layer representation}
We computed the first 100 eigenfrequencies for the domain in Fig.~\ref{fig: ns}(a)
using both the Boyd's method and the standard SVD method
as shown on the first two rows of table \ref{tab:t}, respectively. For both methods, 
the initial number of quadrature nodes is scaled by setting $N = \max{(150, 100+5\kappa)}$.
For the Boyd's method, the initial interval used was $[2, 5]$, converged $\kappa_{100} = 20.4300941760382$
and the largest $N$ is 202.
For moderate eigenfrequencies, as shown on the last two rows of table \ref{tab:t},
we solved the 6 eigenfrequencies in the interval $[100, 100.1]$ using 750 quadrature nodes
using both methods. We used a pure double-layer potential ($\eta=0$)
since this domain is simply-connected and has no problem with
exterior resonances. The error parameter from Section~\ref{s:boyd} is set to $\beta=10^{-14}$.

For the Boyd's method,
the error $\epsilon$ of each eigenfrequency is estimated using the
magnitude of the imaginary part of
the root found, as explained in Sec.~\ref{s:boyd}.
For the SVD method, error $\epsilon$ is estimated as follows.
From Theorem~1 in \cite{bnds}, the distance of any fixed $\kappa_{0}^{2}$ to the true spectrum
can be bounded by $C\kappa_{0} t[u]$, where $C$ is a constant depending only on $\Omega$, $u$ is a solution to \eqref{helm} with $\kappa = \kappa_{0}$,
and $t[u] := \| u \|_{L^{2}(\partial \Omega)}/ \| u \|_{L^{2}(\Omega)}$
is a measure of the relative boundary error.
Since our domain is star-shaped, we can use \cite[(6.1)--(6.2)]{bnds}
to give an explicit estimate for $C$ of approximately 3.5.
By representing $u$ as double layer potential with density $\varphi$, we have $u |_{\partial \Omega} = (D - \frac{1}{2}) \varphi$ and $ u |_{\Omega} = \mathcal{D}\varphi$. Numerically $t[u]$ can be bounded by $\frac{\sigma_{min}(I-M_{N})}{2\| \mathcal{D}\hat{\varphi}\|_{L^{2}(\Omega)}}$, where $\hat{\varphi}$ is the associated right singular vector of $\sigma_{min}(I-M_{N})$. Thus we estimate the relative error in $\kappa$ to be $\frac{C\sigma_{min}(I-M_{N})}{2\kappa \| \mathcal{D}\hat{\varphi}\|_{L^{2}(\Omega)}}$, where $ \| \mathcal{D}\hat{\varphi}\|_{L^{2}(\Omega)}$ is estimated using
crude quadrature scheme in the interior of $\Omega$.
\begin{table}[!ht]
\begin{tabular}{|c| c | c | c | c | c | c | c | c |}
\hline
task & method     & $\max{\operatorname{Im}{\tilde{\kappa}}}$ & mean ${\operatorname{Im}{\tilde{\kappa}}}$ & $\max{\sigma_{\min}}$ & mean $\sigma_{\min}$ &  $\max{\epsilon}$ &mean $\epsilon$ & Time (s)\\ \hline
\multirow{2}{*}{$\kappa \le 20.5$}
&Boyd's  &7.3e-15        & 1.4e-15           & 1.7e-14        & 2.1e-15        &  3.8e-14 & 6.2e-15 & 20\\ \cline{2-9}
&SVD &      -            &       -              & 6.8e-11        & 1.6e-12     &1.1e-10  & 2.6e-12      & 42\\ \hline
\multirow{2}{*}{$\kappa \sim 100$}
&Boyd's  &1.6e-15        & 7.4e-16           & 6.1e-15        & 3.2e-15        &  5.5e-14 & 3.3e-14 & 16\\ \cline{2-9}
&SVD &      -            &       -              & 3.1e-11        & 5.5e-12     &1.1e-11  & 2.0e-12      & 151\\ \hline
\end{tabular}
\caption{Performance data for the nonsymmetric domain in Fig.~\ref{fig: ns}(a)}
\label{tab:t}
\end{table}
\vspace{-0.2in}

\subsection{Crescent-shaped domain solved via the CFIE}
For an example requiring the combined field potential
for a robust solution,
we test the highly-resonant crescent domain in Fig.~\ref{fig:annularSector}(b). Computation is done again for the first 100 eigenfrequencies.
In both methods, the number of quadrature nodes is given by $N = \max{(350, 100+7\kappa)}$.
For the Boyd's method, the initial interval used was $[15, 17]$,
converged $\kappa_{100} = 50.17535680154$
and the largest $N$ is 456.
The error parameter is set to $\beta=10^{-12}$.

For error estimate, the $C$ value for this highly concave domain is not known but we expect it to be $O(1)$ based on discussion in \cite{bnds}.
Thus we computed $\frac{\sigma_{min}(I-M_{N}-i\eta Q_{N})}{2\| \mathcal{D}\hat{\varphi}\|_{L^{2}(\Omega)}}$,
where $\hat{\varphi}$ is the associated right singular vector of $\sigma_{min}(I-M_{N}-i\eta Q_{N})$, as an estimate for
the relative error $\epsilon$ in $\kappa$, up to the constant factor $C$.
\begin{table}[!ht]
\begin{tabular}{| c | c | c | c | c | c | c | c |}
\hline
method     & $\max{\operatorname{Im}{\tilde{\kappa}}}$ & mean ${\operatorname{Im}{\tilde{\kappa}}}$ & $\max{\sigma_{\min}}$ & mean $\sigma_{\min}$ &  $\max{\epsilon / C}$ &mean $\epsilon/ C$ & Time (s)\\ \hline
 Boyd's  & 6.7e-13        & 1.7e-14            & 4.9e-13       & 1.6e-14         & 2.1e-13 & 9.0e-15 & 98 \\ \hline
 SVD & -                     & -                         & 3.5e-6       & 5.0e-8       & 1.4e-11& 1.7e-13  & 368 \\ \hline
\end{tabular}
\caption{Performance data for the crescent domain in Fig.~\ref{fig:annularSector}(b)}
\end{table}
\vspace{-0.2in}
\begin{remark}
Boyd's rooting search method is sufficient to find the first 100 eigenfrequencies to at least 12 digits accuracy for those two examples, 
i.e., adjacent roots were never closer than $10^{-3}$ so the SVD was never needed
to replace Boyd's method.
\end{remark}
Finally, we show some eigenmodes of the crescent domain in Fig.~\ref{fig:annularSector}(b),
computed as follows.
Once we obtain an eigenfrequency $\kappa_{j}$, we can extract the normal
derivative data
from the left kernel of the Nystr\"om matrix $I-M_{N}(\kappa_{j})$
then use Green's representation
formula \eqref{GRF} to reconstruct the eigenmode.
Fig.~\ref{f:modes} shows the first 100 such modes;
they are close to the separation-of-variable forms
which would result for an annular sector.
\begin{figure}[!ht]
\includegraphics[width=0.9\textwidth]{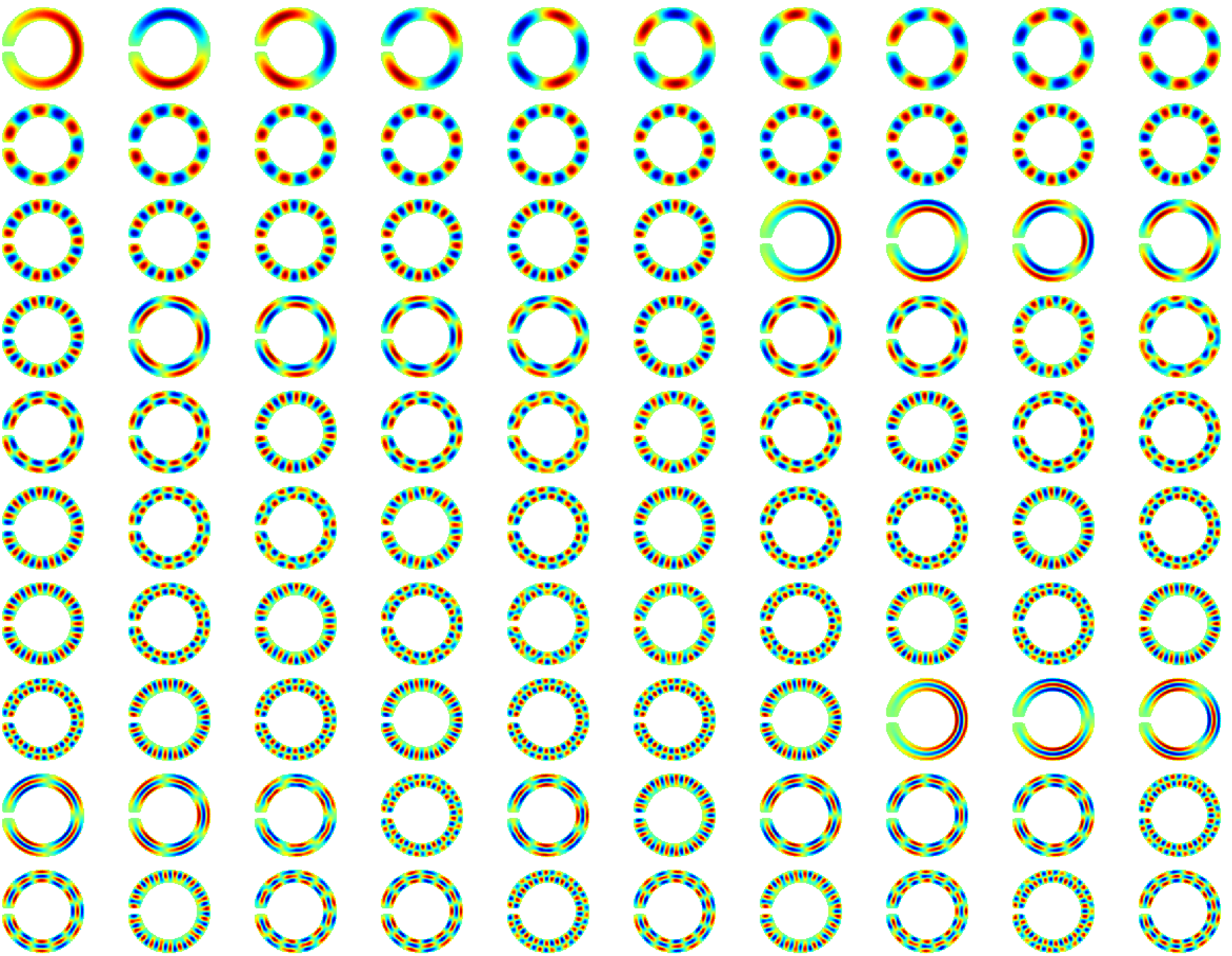}
\caption{%
Modes $u_1$ to $u_{100}$ of the crescent domain, computed via the CFIE
method of this paper, as discussed in Sec.~\ref{s:ps}.
\label{f:modes}
}%
\end{figure}

\section{Conclusions}
\label{s:c}
We have developed a robust method to compute Dirichlet eigenvalues for 2D domains with high accuracy and high efficiency compared to the traditional SVD root-finding method.
We applied Boyd's root-finding method, exploiting the analyticity with respect to frequency of the Fredholm determinant of the boundary integral operator.
This is approximated by the determinant of a Nystr\"om matrix
derived using as spectrally-accurate product quadrature.
Since the determinant is cheap to evaluate, and Boyd's
method requires only around 5 evaluations per eigenvalue found,
we show that the method is 2-10 times faster than existing
SVD-based methods.

In the case of an analytic boundary, we proved that our determinant has
exponential convergence to zero at the true eigenvalues,
and show that this rapid convergence carries over to the computed eigenvalues.
Hence we are able to achieve 13 digits of relative accuracy for
all eigenvalues computed for a star-shaped domain and 12 digits for a highly concave domain, with small numbers of boundary nodes.
For multiply-connected domains or those with exterior resonances,
we introduce a combined-field representation, prove that it
is robust, and show that it eliminates spurious solutions that are
present in the standard approach.
In the case of close eigenfrequencies, we revert to the SVD-based method;
this is not a common occurrence.

We expected that corners, and thus very general domains, can be handled
with a corner-refined quadrature scheme.
One challenge remaining is to analyze a regularization of the CFIE
(case $\eta>0$) in which the Fredholm determinant is not infinite;
the $S$ operator we currently use in the CFIE is not in trace class.
For this we suggest considering $\mathcal{D}+i\eta \mathcal{S}^{2}$.


\bibliographystyle{abbrv}
\bibliography{alex}


\end{document}